\documentclass{article}
\usepackage{amsmath,amsfonts,amssymb,amsthm}
\usepackage{latexsym}
\usepackage{epsfig,overpic}
\usepackage{stmaryrd}
\usepackage{color}
\usepackage{graphicx}
\usepackage{subcaption}
\usepackage{pgfplots}
\newtheorem{lemma}{Lemma}[section]

\newtheorem{remark}{Remark}[section]

\providecommand{\keywords}[1]{\textbf{\textit{Keywords:}} #1}

\let\pa\partial

\def\ds{\mathrm{d}\mathbf{s}}
\newcommand{\R}{\mathbb R}

\newcommand{\bP}{\mathbf P}

\newcommand{\bn}{\mathbf n}

\newcommand{\bp}{\mathbf p}

\newcommand{\br}{\mathbf r}
\newcommand{\bs}{\mathbf s}

\newcommand{\bx}{\mathbf x}

\newcommand{\F}{\mathcal F}

\newcommand{\T}{\mathcal T}

\newcommand{\cO}{\mathcal O}

\newcommand{\Div}{\mbox{\rm div}}

\newcommand{\unablah}{{\nabla}_{\Gamma_h}}

\newcommand{\TGh}{{\mathcal T}_h^\Gamma}
\newcommand{\tnorm}[1]{|\!|\!| #1 |\!|\!|}

\newcommand{\nablaG}{\nabla_{\Gamma}}

\definecolor{darkred}{rgb}{.7,0,0}

\definecolor{green}{rgb}{0,0.7,0}
\newcommand\Green[1]{\textcolor{black}{#1}}

\newcommand\VY[1]{\textcolor{black}{#1}}

\def\d{\hspace{2pt}{\rm d}}

\def\nat{\nabla_{\Gamma}}
\def\nath{\nabla_{\Gamma_h}}

\begin{document}

\title{An adaptive stabilized trace finite element method for surface PDEs}

\author{Timo Heister\thanks{School of Mathematical and Statistical Sciences, Clemson University, Clemson, SC 29634-0975, USA, heister@clemson.edu, vyushut@clemson.edu.} \and  Maxim A. Olshanskii\thanks{Department of Mathematics, University of Houston, Houston, Texas 77204-3008, USA
maolshanskiy@uh.edu, www.math.uh.edu/\string~molshan}
\and Vladimir Yushutin\footnotemark[1]
}

\maketitle
\pgfplotsset{filter discard warning=false}

\markboth{}{}

\begin{abstract}
The paper introduces an adaptive version of the stabilized Trace Finite Element Method (TraceFEM) designed to solve low-regularity elliptic problems on level-set surfaces using a shape-regular bulk mesh in the embedding space. Two stabilization variants, gradient-jump face and normal-gradient volume, are considered for continuous trace spaces of the first and second degrees, based on the polynomial families $Q_1$ and $Q_2$. We propose a practical error indicator that estimates the `jumps' of finite element solution derivatives across background mesh faces and it avoids integration of any quantities along implicitly defined curvilinear edges of the discrete surface elements.  
For the $Q_1$ family of piecewise trilinear polynomials on bulk cells, the solve-estimate-mark-refine strategy, combined with the suggested error indicator, achieves optimal convergence rates typical of two-dimensional problems. We also provide a posteriori error estimates, establishing the reliability of the error indicator for the  $Q_1$  and  $Q_2$ elements and for two types of stabilization. In numerical experiments, we assess the reliability and efficiency of the error indicator. While both stabilizations are found to deliver comparable performance, the lowest degree finite element space appears to be the more robust choice for  the adaptive TraceFEM framework.
\end{abstract}


\section{Introduction}


The Trace or Cut Finite Element Method is one of the approaches used to approximate surface Partial Differential Equations (PDEs) \cite{olshanskii2017trace, burman2015stabilized}. It falls into the category of geometrically unfitted methods because the domain of a variational problem, a two-dimensional surface denoted as $\Gamma$, is embedded within a three-dimensional triangulated domain $\Omega$ that is a subset of $\mathbb{R}^3$, such as a sufficiently large cube.
Identifying the active mesh, denoted as $\Omega_h\subset\Omega$, and performing local refinement or any other mesh cell updating procedure is straightforward due to the geometrical simplicity. We refer to Figure~\ref{fig:snapshots} for a visual representation. Furthermore, handling data structures on the {octree} mesh $\Omega_h$ can be implemented efficiently and is available in many finite element libraries. This flexibility is one of the advantages of the Trace Finite Element Method (TraceFEM). However, it comes with the cost of constructing quadratures on the intersections of $\Gamma$ with cells from $\Omega_h$. The size and shape of these intersections vary uncontrollably between cells, leading to the necessity for a stabilization term, similar to $s_h$ in equation \eqref{abstract}, in any TraceFEM discretization of surface problems. Several variants of such terms are available in the literature \cite{olshanskii2017trace, larson2020stabilization}, but in this context, we will only consider the `gradient-jump' face stabilization and the 'normal-gradient' volume stabilization. These methods have been successfully used and proven to be practical and robust.

Adaptive strategies within the context of stabilized TraceFEM are not yet well-understood. Previous discussions on adaptivity in the TraceFEM setting can be found in the literature \cite{DemlowOlsh, Chernyshenko2015}. In \cite{DemlowOlsh}, there is no stabilization, and an inferior (as seen in the comparison in \cite{olshanskii2017trace}) 'full-gradient' stabilization is considered in \cite{Chernyshenko2015}. Additionally, both papers only considered piece-wise linear finite element spaces, with \cite{DemlowOlsh} assuming tetrahedral meshes and \cite{Chernyshenko2015} using octree meshes.
We extend the adaptive methodology introduced in \cite{DemlowOlsh} by studying the first and second-order stabilized TraceFEM on octree meshes. Another novel aspect is the consideration of two  stabilizations, namely $s_h^{JF}$ (as defined in \eqref{s_jf}) and $s_h^{NV}$ (as defined in \eqref{s_nv}), in the context of adaptive TraceFEM.

Many mathematical models involving surface PDEs necessitate the use of adaptive numerical methods. For instance, the dynamics of liquid crystal films can give rise to the formation of defects \cite{gharbi2013microparticles, hu2016disclination, koning2016spherical, nestler2022active}. Mathematically, a defect in a liquid crystal film corresponds to low regularity solutions of the governing PDEs on surfaces. From a numerical modeling perspective, this entails the need for adaptive refinement and coarsening as the defect forms and evolves along the film.
The evolution of defects is driven by variations in the energy of the liquid crystal and the mass flow, which are governed by the surface Navier--Stokes equation \cite{nestler2022active}. The necessity of addressing these coupled phenomena numerically serves as motivation for the development of adaptive surface FEMs with both first and second-order polynomial accuracy.

In this paper, our focus is on adaptive strategies for the stabilized TraceFEM applied to the Laplace--Beltrami equation, which serves as a prototypical elliptic problem on a surface $\Gamma$ \cite{Dziuk88}. 
An overview of the motivation and the main results follow. At this point, we will omit certain technical details regarding the geometrical consistency of the adaptive method.
To begin, consider an abstract variational problem on a surface $\Gamma$: given $f\in H^{-1}(\Gamma)$, we seek to find $u\in H^1(\Gamma)$ such that $a(u,v)=\left<f,v\right>$ for all $v \in H^1(\Gamma)$. We assume that the bilinear form $a$ is symmetric and coercive.
In the TraceFEM, the discrete space $V_h$ is defined on a graded, regular bulk mesh $\Omega_h$, and we solve the following \VY{discrete problem}:
\begin{align}\label{abstract}
a(u_h, v_h)+s_h(u_h,v_h)=\left<f,v_h\right>, \quad \forall v_h\in V_h
\end{align}
Here, a stabilization form $s_h$ ensures the algebraic stability of the resulting linear algebraic system. Residual and jump indicators can be derived \cite{bernardi2000adaptive} from the integration by parts in $a(u_h, v_h)$, as done in \cite{DemlowOlsh} for an unstabilized TraceFEM. However, in our case, the stabilization $s_h$ is incorporated into an a posteriori estimate.

We would like to highlight two important aspects of the adaptivity methodology for the method \eqref{abstract}:
\begin{itemize}
       \item The jump indicator requires the construction of non-standard one-dimensional quadratures to handle curved intersections of the surface with faces of bulk cells. The associated implementation burden represents a practical inconvenience of the adaptive TraceFEM approach introduced in \cite{DemlowOlsh}.
    \item  We have observed that the ratio $s_h(u_h,u_h)/a(u_h,u_h)$, where both forms are restricted to a single bulk element, often exhibits significant growth, even for uniformly refined meshes. Consequently, the inclusion of the stabilization term $s_h$ in an error indicator has the potential to compromise its efficiency. 
\end{itemize}

To address the first aspect, we propose an alternative error indicator designed for adaptively refined, graded, octree tessellations of the bulk domain $\Omega$, denoted as $\Omega_h$. This novel indicator is reliable and straightforward to compute, as it eliminates the need for integration over the curved intersections of an implicitly defined surface with two-dimensional faces of the bulk cells. Instead, the indicator incorporates a jump term that only requires the use of a standard 2D quadrature for the faces of the bulk mesh cells. Moreover, for the TraceFEM stabilized with the gradient-jump face stabilization, this term is already an integral part of the method.

As for the second aspect, it is worth noting that the efficiency analysis of TraceFEM indicators remains an open question to the best of our knowledge. To explore this further, we undertake a comprehensive numerical investigation to assess the efficiency of the new indicator.
In the case of stabilized TraceFEM with $Q_1$ finite elements, the indicator is found to be efficient. However, in the $Q_2$ case, efficiency gradually diminishes, although the convergence rates for the adaptive gradient-jump stabilized TraceFEM still appear to remain optimal.

The remainder of this paper is organized as follows:
Section~\ref{s_FEM} introduces the stabilized adaptive TraceFEM along with a new computationally practical indicator.
In Section~\ref{s:analysis}, we provide a proof of the reliability estimate for the indicator.
In Section~\ref{s:num},  the adaptive method is tested numerically for low-regularity solutions to the Laplace--Beltrami equation on the unit sphere. We assess both the reliability and efficiency of the method, considering $Q_1$ and $Q_2$ conforming finite elements defined on octree meshes. Furthermore, we perform experiments using the adaptive TraceFEM with two different stabilizations.

\section{The adaptive trace finite element method} \label{s_FEM}
We are interested in the geometrically unfitted finite element method known as the TraceFEM~\cite{ORG09}.
The method  considered in this section is an extension  of the TraceFEM and stabilization techniques introduced in \cite{Alg1,burman2018cut,grande2018analysis} to hexahedral bulk octree   meshes. After formulation of the method for our model problem, the Laplace--Beltrami equation, we introduce error indicators and an adaptive discretization. 

\subsection{Model problem}\label{s_setup}
Let $\Omega$ be an open domain in $\mathbb{R}^3$ and let $\Gamma\subset \Omega$ be a smooth connected  compact and closed hyper-surface embedded in $\R^3$.
For a sufficiently smooth function $g:\Omega\rightarrow \mathbb{R}$ the tangential derivative
on $\Gamma$ is defined by
\begin{equation*}
	\nabla_{\Gamma} g=\nabla g-(\nabla g\cdot \bn)\bn,\label{e:2.1}
\end{equation*}
where $\bn$ denotes the unit normal to $\Gamma$. Denote  by $\Div_\Gamma=\mbox{tr}(\nabla_\Gamma)$ the surface divergence operator and by $\Delta_{\Gamma}=\nabla_\Gamma\cdot\nabla_\Gamma$ the Laplace--Beltrami operator on $\Gamma$.
The  Laplace--Beltrami equation is a model example of an elliptic PDE posed on the surface $\Gamma$. The equation reads as follows:  find $u:\Gamma \rightarrow \mathbb{R}$ satisfying  
\begin{equation}\label{LB} 
	-\Delta_{\Gamma} u +u =f\quad\text{on}~~\Gamma
\end{equation}

The zero order term is added to avoid non-essential technical details of handling  one-dimensional kernel consisting of all constant functions on $\Gamma$.
 The problem is well-posed in the sense of the weak formulation: 
Given $f\in H^{-1}(\Gamma)$, 
find  $ u\in H^1(\Gamma)$ 
satisfying
\begin{equation}\label{LBw} 
	\int_{\Gamma}\big(	 \nabla_{\Gamma} u\cdot\nabla_{\Gamma}v+uv\big)\,ds =\int_\Gamma f\,v\,ds\, \quad\forall\,v\in  H^1(\Gamma).
\end{equation}
If $f\in L^2(\Gamma)$, then the unique solution satisfies $ u\in H^2(\Gamma)$ and $\|u\|_{H^2(\Gamma)}\le c\|f\|_{L^2(\Gamma)} $ with a constant $c$ independent of $f$; see \cite{Aubin}.

\subsection{Discretization}
We assume an octree cubic mesh $\mathcal{T}_h$ covering the bulk domain $\Omega$.
 In addition, we assume that the mesh
is \emph{gradually} refined, i.e., the sizes of two active (finest level) neighboring cubes differ at most by a factor of 2. Such octree grids are also known as balanced. The method also applies for unbalanced octrees, but our analysis and experiments use balanced grids. The set of all active (finest level) faces is denoted by $\pa\mathcal{T}_h$. The mesh is not aligned with the surface $\Gamma$, which can cut through the cubes with no further restrictions.

By $\Gamma_h$ we denote a given  approximation of $\Gamma$ such that
$\Gamma_h$ is a $C^{0,1}$ piecewise smooth surface without  boundary and $\Gamma_h$ is formed by smooth segments:
\begin{equation} \label{defgammah}
 \Gamma_h=\bigcup\limits_{T\in\mathcal{F}_h} \overline{T},
\end{equation}
where $\mathcal{F}_h=\{T\subset\Gamma_h\,:\, T=\Gamma_h\cap S, ~\text{for}\,S\in \T_h\}$. 
For a given $T\in\F_h$ denote by $S_T$ a cube $S_T\in\mathcal{T}_h$ such that $T\subset S_T$ (if $T$ lies on a side shared by two cubes, any of these two cubes can be chosen as $S_T$).

In practice, we construct $\Gamma_h$ as follows. Assume $\phi$ is a signed distance or general level set function for $\Gamma$ .  We define $\Gamma_h$ as the zero level set of $\phi_h$, a piecewise polynomial interpolant to $\phi$ on $\mathcal{T}_h$: 
\[
{\Gamma}_h:=\{\bx\in\Omega\,:\, \phi_h(\bx)=0 \}.
\]
For geometric consistency, the polynomial degree of $\phi_h$ is the same as the degree of piecewise polynomial functions we use to define trial and test spaces in a finite element formulation.  
In some applications, $\phi_h$ is recovered from a solution of a discrete indicator function equation (e.g. in the level set or the volume of fluid methods), without any direct knowledge of $\Gamma$.
  Assumptions of how well  ${\Gamma}_h$ should approximate $\Gamma$ will be given later.

The unit (outward pointing) normal to \VY{$\Gamma_h$} vector $\bn_{h}\VY{=\nabla \phi_h/|\nabla \phi_h|}$ is defined almost everywhere on \VY{$\Omega$}.
We also define
$\bP_h(\bx):=\mathbf{I}-\bn_{h}(\bx)\bn_{h}(\bx)^T$ for $\bx\in \Gamma_h,~\bx$  not on an edge.
The tangential derivative along $\Gamma_h$ is given by $\nabla_{\Gamma_h} g=\bP_h\nabla g$ for sufficiently smooth $g$ defined in a neighborhood of $\Gamma_h$.
\smallskip

Consider a subdomain $\omega_h$ of $\Omega$ consisting only of those end-level cubic cells that contain $\Gamma_h$:
\begin{equation} \label{defomeg}
	\omega_h= \bigcup_{S\in\TGh} S,\quad\text{with}~\TGh=\{S\in\T_h\,:\, S=S_T~\text{for}~ T \in \mathcal{F}_h \}.
\end{equation}
The piecewise constant function $h_S:\omega_h\rightarrow \R$ denotes the bulk cubic cell size.
 Denote by $\Sigma_h$ the set of all end-level internal faces of $\TGh$, i.e. square faces between intersected cells from $\TGh$,
\begin{equation} \label{defsigma}
	\Sigma_h= \{F\in\pa\mathcal{T}_h\,: F\in \textrm{int}(\omega_h) \}.
\end{equation}
The piecewise constant function $h_F:\Sigma_h\rightarrow \R$ denotes the face size. Since the mesh is gradually refined, $h_F=\min(h_{S_F^+},h_{S_F^-})$, where $S_F^+,S_F^-\in \TGh$ are the two bulk cells which share the end-level face $F\in \Sigma_h$.

We are also interested in the set of all faces which are intersected by $\Gamma_h$,
\begin{equation} \label{defsigmaG}
	\Sigma^\Gamma_h= \{F\in\Sigma_h\,: F \cap \Gamma_h\neq \emptyset \,  \}.
\end{equation}
Intersected faces are necessary internal, so that $\Sigma^\Gamma_h \subset \Sigma_h$, but the opposite inclusion does not hold. 

For each cell $S$, let $M_S$ be the affine mapping from the reference unit cube. Then the finite element space of order $k$ is  defined as :
\begin{align}\label{traceVhk}
 V_h^k:=\{v\in C(\omega_h)\ |\ v|_{S}\circ M_S\in Q_k\,, \forall\ S \in\TGh\},
\end{align}
where $Q_k$ is the Lagrangian finite element basis of degree $k$. In case of $k=1$, $V_h=V_h^1$ is the space of piecewise trilinear functions corresponding to the family
 \begin{align}\label{q1_family}
 Q_1=\mbox{span}\{1,x_1,x_2,x_3,x_1x_2,x_1x_3,x_2x_3,x_1x_2x_3\}.
 \end{align}
Note that we consider  $H^1$-conforming (i.e., continuous) finite elements. In this paper we restrict to $k=1,2$.
 
 Let $f^e$ be an extension of $f$ from $\Gamma$ to $\Gamma_h$. The finite element formulation reads: Find $u_h\in V_h^k$ such that
\begin{equation} \label{FEM}
\int_{\Gamma_h}	\big( \nabla_{\Gamma_h} u_h\cdot\nabla_{\Gamma_h} v_h +u_hv_h\big)\,ds +s_h( u_h, v_h) = \int_{\Gamma_h}f^e v_h\,ds\qquad\forall\,v_h\in V_h^k.
\end{equation}
Here $s_h$ is a  stabilization term defined later. The purpose of the stabilization term is to enhance the robustness of the formulation with respect to position of the position of $\Gamma_h$ in the background mesh $\T_h$. In the context of TraceFEM the idea of stabilization was first introduced in \cite{Alg1}.
\smallskip

\subsection{TraceFEM stabilizations}
\label{sec:tracefem_stab}
We are interested in the two commonly used variants of the stabilization terms $s_h$ in \eqref{FEM}. In both cases, the stabilizing term can be assembled elementwise over all end-level cubes intersected by $\Gamma_h$: 
\begin{align*}
s_h( u_h, v_h)=\sum_{S\in\TGh} s_S^\ast( u_h, v_h)\,,\VY{\quad\ast \in \{JF,JF2,NV\}}
\end{align*}
\begin{enumerate}

\item \textbf{Gradient-jump face stabilization} is the  method introduced in \cite{Alg1} following the cutFEM approach developed for the volumetric problems. In the context of the TraceFEM, this stabilization is often used with quasi-uniform bulk meshes, stationary surfaces, and  lowest order elements; see e.g.~\cite{hansbo2015characteristic,burman2016cut,burman2018cut,grande2018analysis}.  

In this variant,  local stabilizing terms are computed over cube's faces which are in the active skeleton $\eqref{defsigma}$,
\begin{align}\label{s_jf}
s_S^{JF}( u_h, v_h)= \sum_{F\in\partial S\cap\Sigma_h}\int_{F} \sigma_F \,\llbracket\nabla  u_h\rrbracket\cdot \llbracket\nabla v_h\rrbracket
\end{align}
where $\sigma_F$ is  $O(1)$ stabilization parameter, and $\llbracket\nabla  u_h\rrbracket = (\nabla  u_h)|_{S^+} - (\nabla  u_h)|_{S^-}$, $F=S^-\cap S^+$, is a ``jump'' of the gradient across the face. 
Note that for continuous FE, stabilization \eqref{s_jf} is equivalent to penalizing the jumps  of normal derivatives across faces. 

 A higher-order version of $s_h^{JF}$ was suggested in \cite{zahedi2017space} and analyzed for quasi-uniform meshes in \cite{larson2020stabilization}. For $Q_2$ elements it reads:
\begin{multline}\label{s_jf2}
s_S^{JF2}( u_h, v_h)=  s_S^{JF}( u_h, v_h) + \int_{\Gamma_h\cap S} \sigma_\Gamma(\bn_h\cdot \nabla{}u_h)( \bn_h\cdot \nabla{}v_h) \\ +\sum_{F\in\partial S\cap\Sigma_h}\int_{F}  \tilde{\sigma}_F{h^2_{F}}\,(\bn_F \cdot \llbracket\nabla^2  u_h\rrbracket \bn_F) (\bn_F\cdot \llbracket\nabla^2 v_h\rrbracket \bn_F)\\
+\int_{\Gamma_h\cap S} \tilde{\sigma}_\Gamma{h^2_{S}}\, (\bn_h\cdot (\nabla{}^2u_h) \bn_h)( \bn_h\cdot (\nabla{}^2v_h)\bn_h),
\end{multline}
where ${\sigma}_\Gamma$, $\tilde{\sigma}_F$, and $\tilde{\sigma}_\Gamma$ are $O(1)$ tuning parameter. The bilinear form $s_h^{JF2}$   stabilizes the trace finite element space $V_h^2$ in the case $Q_2$ polynomial family as shown in  \cite{larson2020stabilization}. In that paper, a more general stabilization $h^{\gamma}s_h^{JF2}$, $0\leq\gamma\leq2$, was considered and the sensitivity of the method to all stabilization parameters was explored. In our numerical results for $Q_2$ family, we choose $\sigma_F=\tilde{\sigma}_F=\tilde{\sigma}_\Gamma=\sigma_\Gamma$.

We see that the gradient-jump stabilization gets quite complicated for higher order elements. Below we consider a normal-gradient volume stabilization, which is universal with respect to the FE degree. 

\item \textbf{Normal-gradient volume stabilization} was introduced in \cite{burman2018cut,grande2018analysis} and it penalizes the variation of the FE solution in the normal direction to the surface. This property was found particularly useful for applying TraceFEM to problems posed on evolving surfaces~\cite{lehrenfeld2018stabilized} and so it is commonly used in this context~\cite{yushutin2020numerical,olshanskii2021finite,olshanskii2022tangential,olshanskii2023eulerian}. \VY{The stabilization reads}:
\begin{align}\label{s_nv}
s_S^{NV}( u_h, {\VY{v_h}})= \int_{S} \rho_S (\bn_h\cdot\nabla  u_h) (\bn_h\cdot\nabla v_h),
\end{align}
where $\rho_S$ is the stabilization parameter, constant in each cell such that 
\[
\rho_S\simeq h_S^{-1}\quad\text{for}~S\in \T_h^\Gamma.
\]
\VY{Note that $\bn_h=\nabla \phi_h/|\nabla \phi_h|$ is well-defined on $\omega_h$ so the integral in \eqref{s_nv} makes  sense.}
    \end{enumerate}

\subsection{Error indicators}
One of the goals of this paper is to construct a new TraceFEM \VY{error} estimator which does not involve complicated and expensive computations on edges $\Gamma_h\cap F$, $F\in{}\Sigma_h^\Gamma$. These edges are available only implicitly as intersections of $\Gamma_h$ with bulk faces. Moreover, one needs to construct an immersed edge quadrature on each intersected face from $\Sigma_h^\Gamma$ which is a significant computational burden. Again, note that some of \VY{the} faces from $\Sigma_h^\Gamma$ are subfaces of bulk cells which complicates the accumulation of flux jumps even further.

To this end, we define the \textit{bulk jump} indicator:
\begin{align}\label{indicator_F}
\eta_F(S_T)=\|\llbracket \nabla u_h \rrbracket\|_{L^2(\pa S_T\cap \omega_h)}\,,\quad S_T\in \TGh.
\end{align}
Note that the indicator \eqref{indicator_F} assesses the variation of the solution gradient across internal, square faces shared by the cubic cells in 
$\TGh$ rather than across the implicit edges  $\Gamma_h\cap F$, $F\in{}\Sigma_h^\Gamma$, as done in 
\cite{DD07,DemlowOlsh,Chernyshenko2015}. The former is more straightforward to compute.  Also note that \eqref{indicator_F} is  accumulated over all faces from \eqref{defsigma} rather then just the intersected faces from  \eqref{defsigmaG}. 

We will also need the \textit{surface residual} indicator,
\begin{align}\label{indicator_R}
\eta_R(T)=h_{S_T}\|f_h+\Delta_{\Gamma_h}  u_h-u_h\|_{L^2(T)}
\end{align}
which was already used in \cite{DD07,DemlowOlsh,Chernyshenko2015}. The computation of the \eqref{indicator_R} requires integration over surface cuts $T=\Gamma_h\cap S_T$, $S_T\in \TGh$ which is a standard procedure in the implementation of TraceFEM \eqref{FEM}.

Thus, for the purpose of local mesh adaptation we use the following error indicator:
\begin{equation}\label{total_indicator}
\eta(S_T):=(\alpha_r\eta_R(T)^2+\alpha_e\eta_F(S_T)^2+\alpha_s s_{S_T}^\ast(u_h,u_h))^{\frac12}.
\end{equation}
with some parameters $\alpha_r,\alpha_e, \alpha_s\ge 0$. 
\begin{remark}\label{rem:omit_bulkjump} \rm
Note that for the gradient-jump face stabilization, the solution's jumps  over faces (i.e. the $\eta_F(S_T)^2$ quantity) are included in $s_{S_T}^\ast(u_h,u_h)$ term  and so the face indicator is extra and we let $\alpha_e=0$ in the cases of $Q_1$ and $Q_2$. Otherwise, in our numerical experiments {with normal-gradient volume}, we choose $\alpha_r=\alpha_e=\alpha_s=1$.
\end{remark}

In this paper we do not consider any indicator of the geometric error resulting from the approximation $\Gamma$ and other geometric quantities. They are assumed to be of a higher order with respect to $h_S$. 

\smallskip
Results of  experiments in Section~\ref{s:num} show that the trace FE adaptive method
based on $\eta(T)$ results in the optimal convergence of the adaptive method in $H^1$ and $L^2$ norms.

\section{Reliability}\label{s:analysis}
In this section we prove an \VY{a posteriori} error estimate that implies the reliability the error indicator \eqref{total_indicator}. We start with several preliminaries.  

\subsection{Preliminaries} 

For the surface $\Gamma$, we consider its neighborhood:
\begin{equation}
\cO(\Gamma):=\{\mathbf{x}\in \mathbb{R}^3\ |\ \mathrm{dist}(\mathbf{x},\Gamma)< \tilde{c}\},\label{e:3.1}
\end{equation}
with a suitable  $\tilde{c}$ depending on $\Gamma$
such that $\omega_h\subset \cO(\Gamma)\subset\Omega$ and the normal projection  $\bp: \cO(\Gamma)\rightarrow\Gamma$,
\begin{equation*}
 \bp(\bx)=\bx-d(\bx)\bn(\bx)
\end{equation*}
is well-defined. Hereafter  $d\in C^2(\cO(\Gamma))$ denotes the signed distance function such that $d<0$ in the interior of $\Gamma$ and $d>0$ in the exterior, and  $\mathbf{n}(\mathbf{x}):=\nabla d(\mathbf{x})$ for all $\mathbf{x}\in \cO(\Gamma)$. Hence,  $\mathbf{n}$ is the normal vector on $\Gamma$ and $|\mathbf{n}(\bx)|=1$ for all $\bx\in \cO(\Gamma)$.  The Hessian of $d$ is denoted by
\begin{equation*}
 \mathbf{H}(\bx):= \nabla^2 d(\bx)\in \mathbb{R}^{3\times 3},\quad \bx\in \cO(\Gamma).
\end{equation*}
The eigenvalues of $\mathbf{H}(\bx)$ are the principal curvatures $\kappa_1(\bx)$, $\kappa_2(\bx)$, and~$0$.

We  assume the following estimates on how well $\Gamma_h$ approximates $\Gamma$:
\begin{align}
&\mathrm{ess\ sup}_{\bx\in \Gamma_h}|d(\bx)| \le c_1 h^{k+1}, \label{e:3.9}\\
&\mathrm{ess\ sup}_{\bx\in \Gamma_h}|\bn(\bx)-\bn_{h}(\bx)| \le c_2   h^k,\label{e:3.10}
\end{align}
with constants $c_1$, $c_2$ independent of $h$ and $k\in\{1,2\}$ in the FE degree.
The assumption is reasonable if $\Gamma$ is defined as the zero level of a (locally) smooth level set function $\phi$
and $\Gamma_h$ is the zero of an  $\phi_h\in V_h$,
where $\phi_h$ interpolates $\phi$ and it holds
\[
\|\phi-\phi_h\|_{L^\infty(\cO(\Gamma))}+h\|\nabla(\phi-\phi_h)\|_{L^\infty(\cO(\Gamma))}\lesssim h^{k+1}.
\]
Here and in the remainder, $A\lesssim B $ means $A\leq c\, B $ for some positive constant $c$ independent of the number of refinement levels and the position of $\Gamma_h$ in the background mesh.

\smallskip
For $\bx\in\Gamma_h$, define
 $\mu_h(\Gamma)(\bx) = (1-d(\bx)\kappa_1(\bx))(1-d(\bx)\kappa_2(\bx))\bn^T(\bx)\bn_h(\bx)$.
The surface measures  $\ds$ and $\ds_{h}$ on $\Gamma$ and $\Gamma_h$, respectively, are related \VY{\cite{DD07}} by
\begin{equation}
 \mu_h(\Gamma)(\bx)\ds_h(\bx)=\ds(\bp(\bx)),\quad \bx\in\Gamma_h. \label{e:3.16}
\end{equation}

The solution of  the Laplace--Beltrami problem and its data are defined on $\Gamma$,
while the finite element method is defined on $\Gamma_h$.
Hence, we need a suitable extension of a function from $\Gamma$ to its neighborhood.
For a function $v$ on $\Gamma$ we define
\begin{equation*}
 v^e(\bx):= v(\bp(\bx)) \quad \hbox{for all } \bx\in \cO(\Gamma).
\end{equation*}
The following formulas for this extended function are well-known (cf. section 2.3 in \cite{DD07}):
\begin{align}
 \nabla u^e(\mathbf{x}) &= (\mathbf{I}-d(\mathbf{x})\mathbf{H})\nabla_{\Gamma} u(\bp(\bx)) \quad \hbox{ in } \cO(\Gamma),\label{grad1}\\
 \nabla_{\Gamma_h} u^e(\bx) &= \bP_h(\bx)(\mathbf{I}-d(\mathbf{x})\mathbf{H})\nabla_{\Gamma} u(\bp(\bx)) \quad \hbox{ a.e. on } \Gamma_h,\label{grad2}
\end{align}
with $\mathbf{H}=\mathbf{H}(\mathbf{x})$. For $\bx\in \Gamma_h$ also define $\tilde{\bP}_{h}(\bx)= \mathbf{I}-\bn_h(\bx)\bn(\bx)^T/(\bn_h(\bx)\cdot\bn(\bx))$.
 One can represent the surface gradient of $u\in H^1(\Gamma)$ in terms of $\nath u^e$ as follows
\begin{equation} \label{hhl}
 \nat u(\bp(\bx))=(\mathbf{I}-d(\bx)\mathbf{H}(\bx))^{-1} \tilde{\bP}_{h}(\bx) \nath u^e(\bx)~~ \hbox{ a.e. }\bx\in \Gamma_h.
\end{equation}
Due to \eqref{e:3.16} and \eqref{hhl}, one gets
\begin{equation}\label{aux13}
 \int_{\Gamma}\nat u\nat v\, \ds=\int_{\Gamma_h} \mathbf{A}_h\nath u^e\nath v^e \, \ds_h \quad \hbox{for all } v\in H^1(\Gamma),
\end{equation}
with $\mathbf{A}_h(\bx)=\mu_h(\bx) \tilde{\bP}^T_h(\bx)(\mathbf{I}-d(\bx)\mathbf{H}(\bx))^{-2}\tilde{\bP}_h(\bx)$.

For  sufficiently smooth $u$ and $|\mu| \leq 2$, it holds (cf. Lemma 3 in \cite{Dziuk88}):
 \begin{equation}
  |D^{\mu} u^e(\bx)|\lesssim \left(\sum_{|\mu|=2}|D_{\Gamma}^{\mu} u(\bp(\bx))| + |\nabla_{\Gamma} u(\bp(\bx))|\right)\quad \hbox{in } \cO(\Gamma),
\label{e:3.13}
 \end{equation}
 
We need the following uniform trace inequalities. For any end level cell $S\subset\omega_h$ and its face $F\subset S$ it holds 
\begin{align}\label{eq_trace}
\|v\|_{L^2(S\cap \Gamma_h)}^2\lesssim h^{-1}_S\|v\|_{L^2(S)}^2+h_S\|\nabla v\|_{L^2(S)}^2\quad \forall~v\in H^1(S).\\
\label{eq_trace2}
\|v\|_{L^2(F\cap \Gamma_h)}^2\lesssim h^{-1}_F\|v\|_{L^2(F)}^2+h_F\|\nabla v\|_{L^2(F)}^2\quad \forall~v\in H^1(F).
\end{align}
Note that for graded octree meshes it holds $h_F\simeq h_S$. 
The proof of \eqref{eq_trace} follows by subdividing any cubic cell into a finite number of regular tetrahedra
and further applying Lemma~4.2 from \cite{Hansbo2003} on each of these tetrahedra. Similar  procedure is applied to prove \eqref{eq_trace2}. 

We will use the following notation
\[
a_h(u,v):= \int_{\Gamma_h}	 (\nabla_{\Gamma_h} u\cdot\nabla_{\Gamma_h} v +uv)\,d\bs_h.
\]

\subsection{A posteriori estimate}\label{s_adapt} 
In this section, we deduce  an a posteriori error estimate for the TraceFEM ~\eqref{FEM}. 
For the sake of analysis we make the following \emph{assumptions}:  
\begin{description}
\item{(i)} The octree mesh is \VY{gradually} refined; 
\item{(ii)} For any $\bs\in \Gamma$ denote by $K(\bs)$ a number of end-level cubic cells from $\omega_h$ intersected by the line $\ell(\bs)=\{\bx\in\mathcal{O}(\Gamma):\,\bp(\bx)=\bs\}$. We assume $K(\bs)\le K$ with a constant $K$ independent of $\bs$ and the number of refinement levels.  
\end{description}

\VY{In practice, the first assumption can be satisfied by triggering the refinement of any cell which has a finer neighbor already marked for refinement. The second assumption does not pose any practical restrictions  and in experiments we observed that $K(\bs)$ is small for all $\bs$ sampled for testing. An explanation of 
why Assumption (ii) is reasonable relies on the smoothness of $\Gamma$ and the use of gradually refined meshes. Indeed for a $C^2$ surface, one can choose such $O(1)$ neighborhood  $\mathcal{O}(\Gamma)$ that $\ell(\bs)$ intersects $\Gamma$ only once at point $\bs$. Assume $K(\bs)\to\infty$ with a mesh refinement for some $\bs\in\Gamma$. Since the end-level cells are getting arbitrary small, this implies that $\ell(\bs)$ intersects or touches $\Gamma$ at the point of accumulation. However, such only point can be $\bs$ and  $\ell(\bs)\perp \Gamma(\bs)$ while $\Gamma$ is increasingly flat in the local (mesh) scale. For a graded mesh, this may result only in a finite number  of intersected end-level cells. 
}

Consider the surface finite element error $e_h=u^e-u_h$ in $\omega_h$. By $e_h^l$ we denote the lift of the error function on $\cO(\Gamma)$, $e_h^l(\bx)=u(\bp(\bx))-u_h(\bs)$ with $\bs\in\Gamma_h$ such that $\bp(\bs)=\bp(\bx)$.
Note that $e_h^l$ is constant in normal directions to $\Gamma$, i.e. $e_h^l=(e_h^l|_\Gamma)^e$.
Further we prove an a posteriori bound for the augmented $H^1$-norm of $e_h^l$ on $\Gamma$, i.e. for 
\begin{equation}
\label{3.1}
\tnorm{e_h}^2= a( e_h^l,  e_h^l)+s_h(e_h, e_h),\quad\text{with}~a( u,v)=\int_\Gamma(\nabla_\Gamma u\cdot \nabla_\Gamma v + uv )\,ds.
\end{equation}
Using straightforward calculations and \eqref{aux13} one checks the following identities for any $\psi_h\in V_h$
\begin{equation}\label{s4_e1}
\begin{split}
\tnorm{e_h}^2&= \int_\Gamma f e_h^l\,d\bs-a(u_h^l,   e_h^l)+s_h(e_h, e_h)\\
&= \int_{\Gamma_h} f^e e_h\mu_h\,d\bs_h-\int_{\Gamma_h} f_h \psi_h\,d\bs_h+a_h(u_h,  \psi_h)+s_h(u_h,  \psi_h)-a(u_h^l,   e_h^l)+s_h(e_h, e_h)\\
&=\int_{\Gamma_h} (f^e\mu_h-f_h) e_h\,d\bs_h +\int_{\Gamma_h} f_h(e_h-\psi_h)\,d\bs_h+a_h(u_h,  \psi_h- e_h) +s_h(u_h, \psi_h-e_h)
\\&\qquad -\int_{\Gamma_h}(\mathbf{A}_h-\bP_h)\nabla_{\Gamma_h} u_h \cdot \nabla_{\Gamma_h} e_h\, \ds_h.
\end{split}
\end{equation}

Element-wise integration by parts for the third term on the right hand side of \eqref{s4_e1} gives
\begin{equation}\label{s4_e2}
a_h(u_h,  \psi_h-e_h)= \int_{\Gamma_h} (\Delta_{\Gamma_h}  u_h-u_h)( e_h-\psi_h) \d \bs_h
 -\frac12\sum_{T\in \F_h}\int_{\partial T}
\llbracket \unablah u_h \rrbracket ( e_h-\psi_h) \d\br.
\end{equation}
The Cauchy inequality gives
\[
s_h(u_h, \psi_h-e_h)\le \left(\sum_{S\in \TGh} s_S^\ast(u_h,u_h)\right)^{\frac12}
\left(\sum_{S\in \TGh} s_S^\ast(\psi_h-e_h,\psi_h-e_h)
\right)^{\frac12}.
\]

 Substituting \eqref{s4_e2} into \eqref{s4_e1} and applying the Cauchy inequality elementwise over $\F_h$ to estimate integrals, we get
\begin{equation}\label{s4_e3}
\begin{split}
\tnorm{e_h}^2&\lesssim
\sum_{T\in \F_h} \left(\|f^e\mu_h-f_h\|_{L^2(T)}+
\|\mathbf{A}_h-\bP_h\|_{L^\infty(T)}\|\nabla_{\Gamma_h} u_h\|_{L^2(T)}\right)\| e_h\|_{H^1(\Gamma_h)} 
\\&+
\left(\sum_{T\in \F_h}\eta_R(T)^2\right)^{\frac12}
\left(\sum_{T\in \F_h}h^{-2}_{S_T}\| e_h-\psi_h\|_{L^2(T)}^2
\right)^{\frac12}\\
&+
\left(\sum_{T\in \F_h} h_{S_T}\|\llbracket \unablah u_h \rrbracket\|^2_{\partial T} \right)^{\frac12}
\left(\sum_{T\in \F_h} h_{S_T}^{-1}\| e_h-\psi_h\|_{L^2(\partial T)}^2
\right)^{\frac12} \\
&+
\left(\sum_{S\in \TGh} s_S^\ast(u_h,u_h)\right)^{\frac12}
\left(\sum_{S\in \TGh} s_S^\ast(\psi_h-e_h,\psi_h-e_h)
\right)^{\frac12}.
\end{split}
\end{equation}
To proceed further we need several results, which we split into a few lemmas.

\begin{lemma}\label{L3} For all $T\in\mathcal{F}_h$ it holds 
\begin{equation}\label{eq:l1_new}
 h_{S_T}\|\llbracket \unablah u_h\rrbracket\|^2_{L^2(\partial T)} \lesssim 
  \eta_F(S_T)^2.
\end{equation}

\end{lemma}
\begin{proof}
Recall that the face-based indicator $\eta_F(S_T)$ for a cell $S_T$ includes all internal faces $F\in \pa{}S_T \cap \Sigma_h$ rather than only  faces from $\pa{}S_T\cap \Sigma_h^\Gamma$.
Also note that $\llbracket \unablah u_h \rrbracket=\llbracket\bP_h\nabla u_h\rrbracket$ is a rational function of a finite degree on each face of $S_T$. Application of the uniform trace estimate~\eqref{eq_trace2} followed by the FE inverse estimate on each face $F\subset \pa{}S_T\cap \Sigma_h^\Gamma$ gives the assertion.
\end{proof}

\begin{lemma} The following bound holds for both stabilizations and FE degrees: 
\begin{equation}\label{eq:s_bound}
 s_S^\ast(\psi_h-e_h,\psi_h-e_h)\lesssim s_S^\ast(e_h,e_h) + h^{-1}_S\|\nabla
\psi_h\|^2_{L^2(\omega(S))}.
\end{equation}
\VY{where $\omega(S)$ denotes a union of cubic cells from $\omega_h$ sharing faces with $S$.}
\end{lemma}
\begin{proof} We first apply the triangle inequality to show
\begin{equation}\label{eq:639}
    s_S^\ast(\psi_h-e_h,\psi_h-e_h)\le 2(s_S^\ast(e_h,e_h) + s_S^\ast(\psi_h,\psi_h))
\end{equation}
We need to estimate the second term on the right-hand side. For the gradient-jump stabilization and $k=2$ we have
\begin{multline}\label{eq:643}
s_S^{JF2}(\psi_h,\psi_h)= \sigma_\Gamma\|\bn_h\cdot \nabla{}\psi_h\|^2_{L^2(\Gamma_h\cap S)}+\tilde{\sigma}_\Gamma{h^2_{S}}\|\bn_h\cdot (\nabla{}^2\psi_h) \bn_h\|^2_{L^2(\Gamma_h\cap S)}  \\ +\sum_{F\in\partial S\cap\Sigma_h}\left( \sigma_F \,\|\nabla  \psi_h\rrbracket\|^2_{L^2(F)}  +  \tilde{\sigma}_F{h^2_{F}}\|\bn_F \cdot \llbracket\nabla^2  \psi_h\rrbracket \bn_F\|^2_{L^2(F)} \right).
\end{multline}

To estimate the first two terms on the right-hand side of \eqref{eq:643}, we apply  the trace estimate \eqref{eq_trace}: 
\begin{equation}\label{eq:648}
\begin{split}
   &\|\bn_h\cdot \nabla{}\psi_h\|^2_{L^2(\Gamma_h\cap S)} \le \|\nabla{}\psi_h\|^2_{L^2(\Gamma_h\cap S)}
   \lesssim h_{S}^{-1}  \|\nabla\psi_h\|^2_{L^2(S)} + h_{S}  \|\nabla^2\psi_h\|^2_{L^2(S)}
   \lesssim h_{S}^{-1}  \|\nabla\psi_h\|^2_{L^2(S)} \\
 {h^2_{S}}&\|\bn_h\cdot (\nabla{}^2\psi_h) \bn_h\|^2_{L^2(\Gamma_h\cap S)} 
 \le {h^2_{S}}\|\nabla{}^2\psi_h\|^2_{L^2(\Gamma_h\cap S)} 
 \lesssim {h_{S}}\|\nabla{}^2\psi_h\|^2_{L^2(S)} \lesssim h_{S}^{-1}\|\nabla\psi_h\|^2_{L^2(S)}.
\end{split}
\end{equation}
To estimate the third and fourth terms on the right-hand side of \eqref{eq:643}, we apply the  finite element trace and inverse inequalities:
\begin{equation}\label{eq:658}
\begin{split}
 \sum_{F\in\partial S\cap\Sigma_h}\sigma_F \,\|\llbracket \nabla  \psi_h\rrbracket\|^2_{L^2(F)} &\lesssim  \|\llbracket \nabla\psi_h\rrbracket\|^2_{L^2(\partial S\cap\Sigma_h)} \lesssim h_{S}^{-1}  \|\nabla\psi_h\|^2_{L^2(\omega(S))}\\
   \sum_{F\in\partial S\cap\Sigma_h}  {h^2_{F}}\|\bn_F \cdot \llbracket\nabla^2  \psi_h\rrbracket \bn_F\|^2_{L^2(F)}& 
 \lesssim  h^2_{S}\|\llbracket\nabla^2  \psi_h\rrbracket\|^2_{L^2(\partial S\cap\Sigma_h)}
  \lesssim  h_{S}\|\nabla^2  \psi_h\|^2_{L^2(\omega(S))}\\
  &\lesssim h_{S}^{-1}  \|\nabla\psi_h\|^2_{L^2(\omega(S))}.
\end{split}
\end{equation}
The combination of \eqref{eq:643}--\eqref{eq:658} gives
\begin{equation}\label{eq:670}
 s_S^{JF2}(\psi_h,\psi_h)\lesssim  h_{S}^{-1}  \|\nabla\psi_h\|^2_{L^2(\omega(S))}.  
\end{equation}
Of course, the same bound \eqref{eq:670} holds also for $k=1$. 
For the normal-volume stabilization we have
\begin{equation}\label{eq:671}
s_S^{NV}(\psi_h,\psi_h)= \rho_S\|\bn_h\cdot \nabla{}\psi_h\|^2_{L^2(S)}\lesssim
 h^{-1}_S\|\bn_h\cdot \nabla{}\psi_h\|^2_{L^2(S)},
\end{equation}
where we used that $\rho_S$ is an $O(h^{-1}_S)$ \VY{parameter}. Substituting \eqref{eq:670} and \eqref{eq:671} \VY{in} \eqref{eq:639} proves the lemma.\\
\end{proof}

 Due to geometric approximation properties \eqref{e:3.9}, \eqref{e:3.10} and ``lifting'' identities \eqref{e:3.16} and \eqref{grad2}
we have
\begin{equation}\label{aux25}
\| e_h\|_{H^1(\Gamma_h)} \lesssim \| e_h^l\|_{H^1(\Gamma)}.
\end{equation}

\begin{lemma} There  exist $\psi_h\in V_h$ such that
\begin{equation}\label{A2}
\sum_{T\in \F_h}\left[h^{-2}_{S_T}\| e_h-\psi_h\|_{L^2(T)}^2
+h_{S_T}^{-1}\| e_h-\psi_h\|_{L^2(\partial T)}^2
+s_S^\ast(\psi_h-e_h,\psi_h-e_h)
\right]\lesssim \tnorm{e_h}^2.
\end{equation}
\end{lemma}

\begin{proof}
 To handle the edge term on the left-hand side of \eqref{A2}, we need some further constructions: For a curved edge $e\subset\partial T$ denote by $F_e\subset \partial S_T$ the face of $S_T$ such that $e\subset F_e$. Denote by $\omega(e)\subset\T_h$ the set of all cubic cells touching $F_e$.  Let $\tilde\phi_h$ be the natural polynomial extension of the level-set function $\phi_h|_{S_T}$ and $\widetilde{\Gamma}_h(e)=\{\bx\in\omega(e)\,:\,\tilde\phi_h(\bx)=0\}$ be a  smooth approximation of $\Gamma$ locally in $\omega(e)$. Note that due to the graded refinement assumption there is a $h_{S_T}/2$ neighborhood of $e$ in $\widetilde{\Gamma}_h(e)$.
Then for $\rho\in H^1(\widetilde{\Gamma}_h(e))$ in holds
\begin{equation}\label{aux590}
\|\rho\|_{L^2(e)}^2 \lesssim h^{-1}_{S_T}\|\rho\|_{L^2(\widetilde{\Gamma}_h(e))}^2 +h_{S_T}\|\nabla_{\widetilde{\Gamma}_h(e)}\rho\|_{L^2(\widetilde{\Gamma}_h(e))}^2.
\end{equation}
The estimate \eqref{aux590} follows from a standard flattening argument and applying a trace inequality as in \eqref{eq_trace2}.

We apply the bulk and \eqref{eq_trace}  trace inequalities and \eqref{aux590}  to estimate
\begin{equation}\label{A2_1}
\begin{split}
h^{-2}_{S_T}\| e_h-\psi_h\|_{L^2(T)}^2&+\sum_{e\in\partial T}h_{S_T}^{-1}\| e_h-\psi_h\|_{L^2(e)}^2\\ &\lesssim
h^{-3}_{S_T}\| e_h^l-\psi_h\|_{L^2(S_T)}^2+h^{-1}_{S_T}\|\nabla( e_h^l-\psi_h)\|_{L^2(S_T)}^2\\
&\qquad+\sum_{e\in\partial T}\left(h_{S_T}^{-2}\| e_h^l-\psi_h\|_{L^2(\widetilde{\Gamma}_h(e))}^2+\|\nabla_{\widetilde{\Gamma}_h}( e_h^l-\psi_h)\|_{L^2(\widetilde{\Gamma}_h(e))}^2\right)\\
&\lesssim
h^{-3}_{S_T}\| e_h^l-\psi_h\|_{L^2(\omega(e))}^2+h^{-1}_{S_T}\|\nabla( e_h^l-\psi_h)\|_{L^2(\omega(e))}^2 \\
&\qquad +\sum_{e\in\partial T}\left(\|\nabla_{\widetilde{\Gamma}_h} e_h^l\|_{L^2(\widetilde{\Gamma}_h(e))}^2+\|\nabla\psi_h\|_{L^2(\widetilde{\Gamma}_h(e))}^2\right)\\
&\lesssim
h^{-3}_{S_T}\| e_h^l-\psi_h\|_{L^2(\omega(e))}^2+h^{-1}_{S_T}\|\nabla( e_h^l-\psi_h)\|_{L^2(\omega(e))}^2\\
&\qquad+\sum_{e\in\partial T}\left(\|\nabla_\Gamma e_h^l\|_{L^2(\bp(\widetilde{\Gamma}_h(e)))}^2+h^{-1}_{S_T}\|\nabla\psi_h\|_{L^2(\omega(e))}^2\right),
\end{split}
\end{equation}
where we used  an estimate
\begin{equation}\label{aux599-2}
\|\nabla_{\widetilde{\Gamma}_h} e_h^l\|_{L^2(\widetilde{\Gamma}_h(e))}\lesssim \|\nablaG e_h^l\|_{L^2(\bp(\widetilde{\Gamma}_h(e)))},
\end{equation}
which holds due to \eqref{e:3.16}, \eqref{grad2} and the fact that \eqref{e:3.9},  \eqref{e:3.10} also hold for the locally extended $\Gamma_h$ with possibly different $O(1)$ constants $c_1$, $c_2$.
Also note that for any lifted function $u^l\in L^2(\omega_h)$
\begin{equation}\label{aux599}
\|u^l\|_{L^2(S)}^2 \lesssim h_{S_T}\|u^l\|_{L^2(\bp(S))}^2.
\end{equation}
Thanks to our assumption (i) there is a Scott-Zhang type interpolant $\psi_h\in V_h$ of $ e_h^l\in H^1(\Omega)$~\cite{heuveline2007h}  such that
\begin{equation}\label{A2_2}
h^{-1}_{S}\| e_h^l-\psi_h\|_{L^2(S)}+\|\nabla\psi_h\|_{L^2(S)}\lesssim \| e_h^l\|_{H^1(\omega(S))}\quad\forall~S\in\Omega_h,
\end{equation}
where $\omega(S)$ is defined as follows: Let $\tilde\omega(S)$ consist of $S$ and of all end-level cubic cells touching $S$, then
$\omega(S)$ is a patch of cells defined as the union of $\tilde\omega(S)$ and of all end-level cubic cells touching $\tilde\omega(S)$.
We assume $\tilde{c}$ in \eqref{e:3.1} to be sufficiently large  and $h$ sufficiently small  that $\omega(S)\subset \cO(\Gamma)$  for all $S\in\Omega_h$.

Applying in \eqref{A2_1} the estimates from \eqref{A2_2}, \eqref{aux599} and the result from Lemma~\ref{L3}  yields
\begin{equation}\label{A2_3}
\begin{split}
\sum_{T\in \F_h}\left[h^{-2}_{S_T}\| e_h-\psi_h\|_{L^2(T)}^2\right. & \left.
+h_{S_T}^{-1}\| e_h-\psi_h\|_{L^2(\partial T)}^2 + s_{S_T}^\ast(\psi_h-e_h,\psi_h-e_h) \right]\\
&\lesssim \sum_{T\in \F_h}\left( h_{S_T}^{-1} \| e_h^l\|_{H^1(\omega(S_T))}^2+\|\nabla_\Gamma e_h^l\|_{L^2(\bp(\omega(S_T)))}^2+ h^{-1}_S\|\nabla
\psi_h\|^2_{L^2(\omega(S_T))}\right)\\
&\lesssim \sum_{T\in \F_h}\left( h_{S_T}^{-1} \| e_h^l\|_{H^1(\omega(S_T))}^2+\|\nabla_\Gamma e_h^l\|_{L^2(\bp(\omega(S_T)))}^2\right)\\ &\lesssim  \sum_{S\in \omega_h} \| e_h^l\|_{H^1(\bp(\omega(S_T)))}^2.
\end{split}
\end{equation}
In the last inequality we also used the fact that for the graded octree mesh $\mbox{diam}(\omega(S_T))\simeq h_{S_T}$.
Due to assumption (i) any cell $S_T$ may belong to a uniformly bounded number of patches. Thanks to this and assumption (ii) any $\bx\in\Gamma$ may belong to the projections of patches which  total number is also  uniformly bounded. This establishes the bound
\begin{equation}\label{A2_4}
\sum_{T\in \F_h} \| e_h^l\|_{H^1(\bp(\omega(S_T)))}^2\lesssim \| e_h^l\|_{H^1(\Gamma)}^2.
\end{equation}
Using   \eqref{A2_1}--\eqref{A2_4} proves the lemma. 
\end{proof}

Combining \eqref{s4_e3}, \eqref{eq:l1_new} and \eqref{aux25}, \eqref{A2} gives the following \textit{a posteriori} error estimate
\begin{center}
\framebox{
\parbox[c]{0.9\textwidth}{
\begin{multline}
\label{apost}
\tnorm{e_h}\lesssim
\left(\sum_{T\in \F_h} \|f^e\mu_h-f_h\|^2_{L^2(T)}+ 
\|\mathbf{A}_h-\bP_h\|^2_{L^\infty(T)}\|\nabla_{\Gamma_h} u_h\|^2_{L^2(T)}
\right)^{\frac12}\\
+
\left(\sum_{T\in \F_h}\left[\eta_R(T)^2+\eta_F(S_T)^2 + s_{S_T}^\ast(u_h,u_h)\right]\right)^{\frac12}.  
\end{multline}  
}
}
\end{center}

Assume that local grid refinement leads to better local surface reconstruction, i.e. \eqref{e:3.9}
and \eqref{e:3.10} can be formulated locally, then it holds $\|f^e\mu_h-f_h\|_{L^2(T)}+\|\mathbf{A}_h-\bP_h\|_{L^\infty(T)}=O(h^{k+1})$. In this case, the first term on the right-hand side of \eqref{apost} is of higher order if $k\ge1$ for $Q_1$ and $k\ge2$ for $k=2$.

\section{Numerical examples}\label{s:num}
This section presents a numerical study of an adaptive version of the stabilized TraceFEM \eqref{FEM}, which relies on the novel indicator \eqref{apost}. First, we provide details of the adaptive algorithm, including the surface approximation, in Section~\ref{sec:adaptiveTrace}. Next, we confirm a posteriori estimates for the families $Q_1$ and $Q_2$. Moreover, we address the efficiency of the indicator using a manufactured solution. We test both gradient jump and normal gradient volume stabilizations. However, we omit the bulk jump indicator $\eta_F(S_T)$ \eqref{indicator_F} in the proposed indicator \eqref{numerical_indicator} if the TraceFEM scheme \eqref{linear_system} is stabilized by including $s_h^{JF}$ or $s_h^{JF2}$ forms; see Remark~\ref{rem:omit_bulkjump}.

\subsection{A low-regularity test case}
This section discusses the model problem \eqref{LBw}, the solution of which is not regular enough to provide optimal rates of convergence if uniform refinement is employed. 
We consider the unit sphere $\Gamma$ and a family of solutions
$u=u_\lambda \in H^{1+\lambda}(\Gamma)$, $0\leq\lambda\leq 1$,  such that
    \begin{equation}
  -\Delta_{\Gamma} u + u=f,\label{eq:lambdaLB}
\end{equation}
with the forcing $f=f_\lambda\in H^{\lambda-1}(\Gamma)$. Consequently, by choosing different values of $\lambda$, we may obtain exact solutions of desired regularity.  
An example  \cite{Chernyshenko2015} of such a family is given in spherical polar coordinates $(\phi, \theta)$, $\theta\in [0,\pi]$, $\phi\in(-\pi,\pi]$, by
\begin{equation}\label{polar}
    u= \sin^\lambda\theta\sin\phi,\qquad  f=(1+\lambda^2+\lambda)\sin^\lambda\theta\sin\phi+(1-\lambda^2)\sin^{\lambda-2}\theta\sin\phi.
\end{equation}

  Clearly,  $u$ and $f$ have singularities at the north, $\theta=0$ or $(x,y,z)=(0,0,1)$, and the south, $\theta=\pi$ or $(x,y,z)=(0,0,-1)$, poles (see Figure~\ref{fig:snapshots}) while being harmonic in the azimuthal direction $\phi$  for each fixed $\theta\neq 0,\pi$.
 
Before the iterative adaptive procedure starts, one constructs a sufficiently fine mesh of $\Omega=[-2,2]^3$ so the initial surface approximation $\Gamma_h$ is well-defined. To this end, the distance function $d(x,y,z)=x^2+y^2+z^2-1$ is chosen for the level-set description of the unit sphere $\Gamma$. The edges of the cube $\Omega$ are  divided in eight equal segments of length $h=0.5$, see Figure~\ref{fig:snapshots}, cycle$=0$. These cells constitute the initial mesh $\mathcal{T}_h$.

\subsection{Adaptive stabilized TraceFEM}
\label{sec:adaptiveTrace}
In this section we present the adaptive algorithm tested in the numerical experiments.
The adaptive procedure is a sequence of cycles each consisting of \VY{the} three steps below.

\textit{Step 1} (\texttt{APPROXIMATE GEOMETRY}). To guarantee continuity of the surface approximation,
{we first resolve all hanging nodes in $\mathcal{T}_h$ by adding a sufficient number of linear constraints}. The interpolant $\phi^k_h$ of order $k$ of the  distance function $d$  on the mesh $\mathcal{T}_h$ identifies the active domain $\omega_h$ consisting of intersected cells $\mathcal{T}_h^\Gamma$. Geometrical information such as the normal vector $\bn_h$ and the surface quadratures representing $\Gamma_h$ is derived from the discrete distance function $\phi_h^k$.

\textit{Step 2} (\texttt{SOLVE}).   The finite element space $V_h^{k}$ consists of continuous piece-wise $Q_1$ or $Q_2$ functions defined on $\mathcal{T}_h^\Gamma$. We solve the following linear system: find $u_h\in V_h^{k}$ such that
  \begin{align}\label{linear_system}
 \int_{\Gamma_h}	 \nabla_{\Gamma_h} u_h\cdot\nabla_{\Gamma_h}v_h+ \int_{\Gamma_h}u_h v_h+{s_h( u_h, v_h)} = \left(f^e, v_h\right)_{\Gamma_h} \,,\qquad\, \forall v_h \in V_h^{k}
	\end{align}
 where the term $s_h$ represents one of stabilizations from Section~\ref{sec:tracefem_stab}. 
 
\textit{Step 3} (\texttt{ESTIMATE\&MARK\&REFINE}). Fix a $0<\theta<1$. Using the discrete solution $u_h$, we compute the indicator $\eta(S_T)$, 
\begin{align}\label{numerical_indicator}
\eta^2(S_T)=\|\llbracket \nabla u_h \rrbracket\|^2_{L^2(\pa S_T\cap \omega_h)}+h^2_{S_T}\|f^e+\Delta_{\Gamma_h}  u_h-u_h\|^2_{L^2(T)} +  s_{S_T}^\ast(u_h,u_h)
\end{align}  on each intersected cell $S_T\in \mathcal{T}_h^\Gamma$. Next we determine the smallest by cardinality set $\mathcal{T}_h^\theta\subset \mathcal{T}_h^\Gamma$ such that
\begin{align}\label{fraction}
    \sum_{S_T \in \mathcal{T}_h^\theta} \eta^2(S_T) > \theta \sum_{S_T \in \mathcal{T}_h^\Gamma} \eta^2(S_T) 
\end{align}
and, finally, refine the cells in $\mathcal{T}_h^\theta$ uniformly.

 This completes the first cycle. At the beginning of the next cycle the new mesh $\mathcal{T}_h$, refined near $\Gamma$, of the domain $\Omega$ is available and we proceed to Step 1.

 \begin{figure}[htb]
\centering
\captionsetup[subfloat]{labelformat=empty}
  \subfloat[cycle=0]{%
    \includegraphics[width=.24\textwidth]{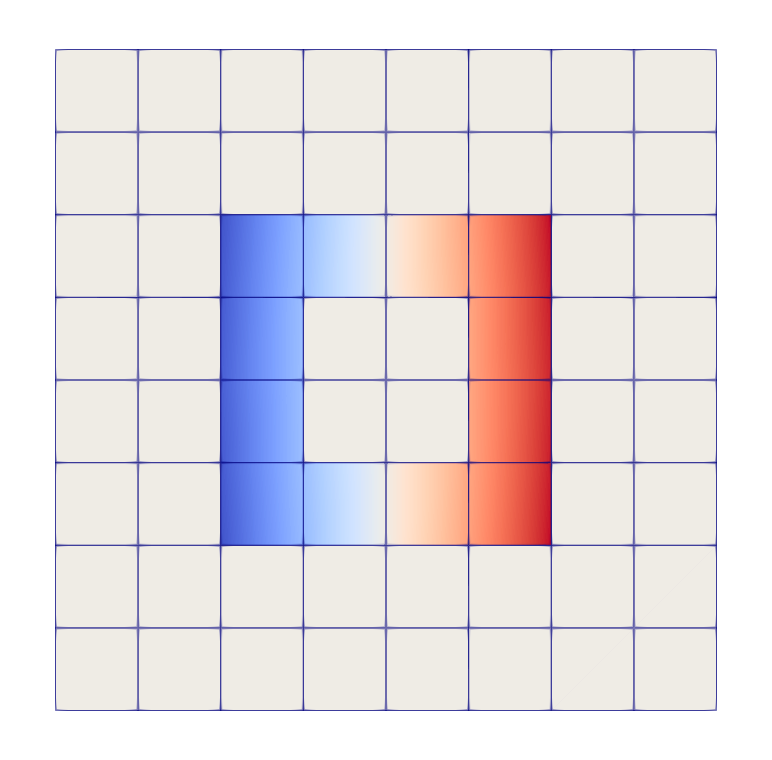}}\hfill
  \subfloat[cycle=1]{%
    \includegraphics[width=.24\textwidth]{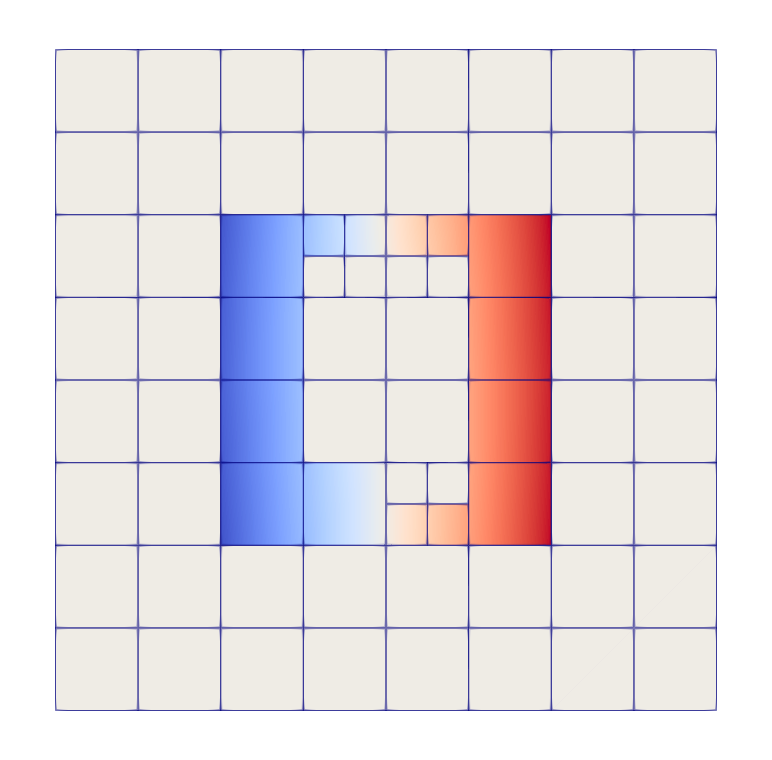}}\hfill
  \subfloat[cycle=2]{%
    \includegraphics[width=.24\textwidth]{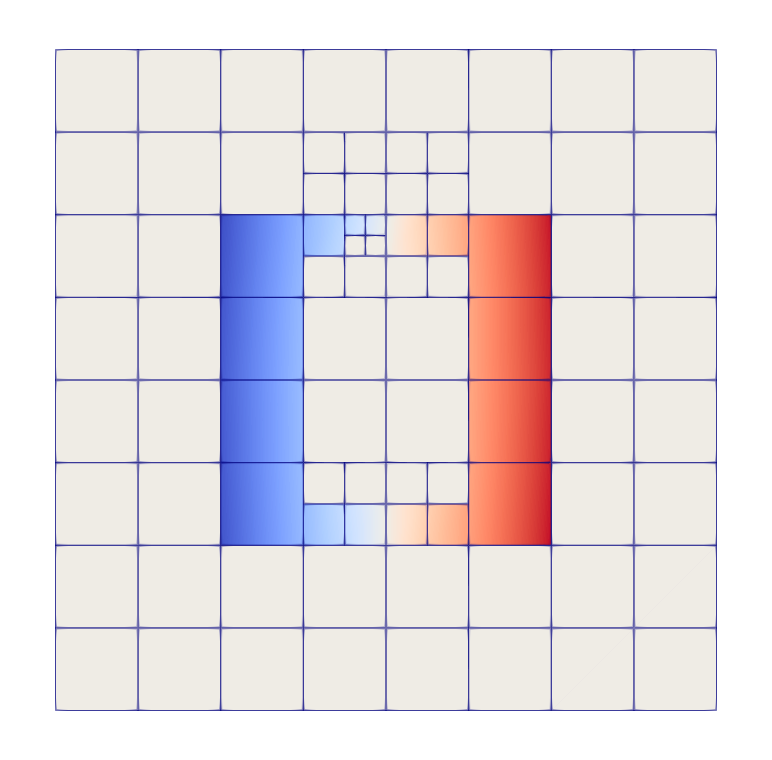}}\hfill
  \subfloat[cycle=11]{%
    \includegraphics[width=.24\textwidth]{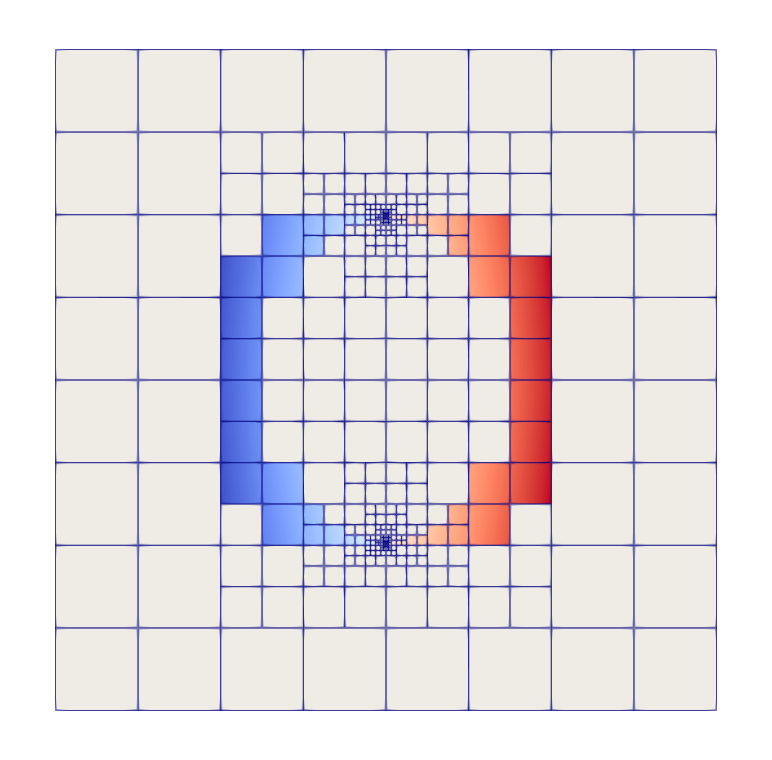}}\\
      \subfloat[north pole, cycle=12]{%
    \includegraphics[width=.24\textwidth]{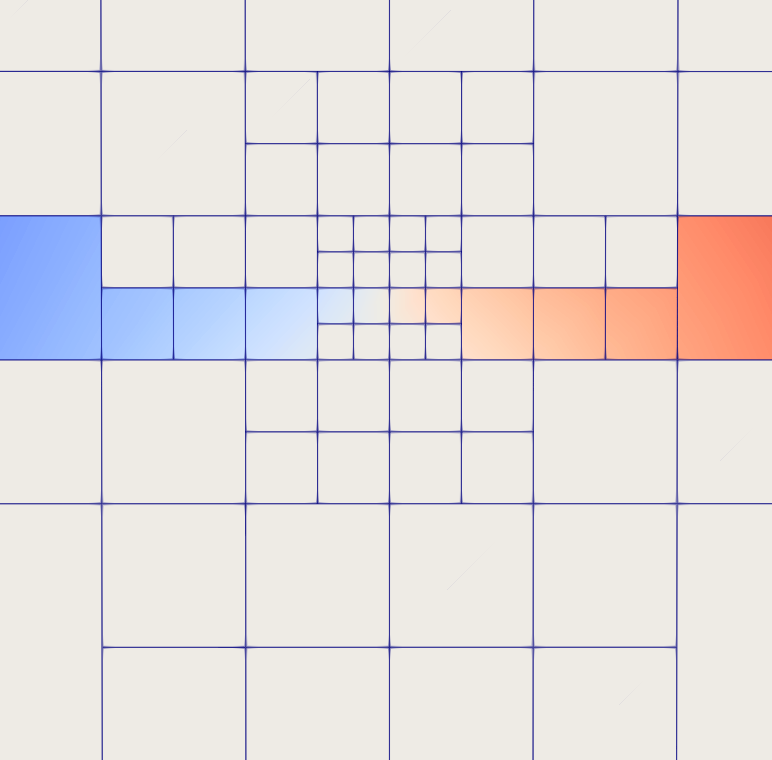}}\hfill
  \subfloat[north pole, cycle=13]{%
    \includegraphics[width=.24\textwidth]{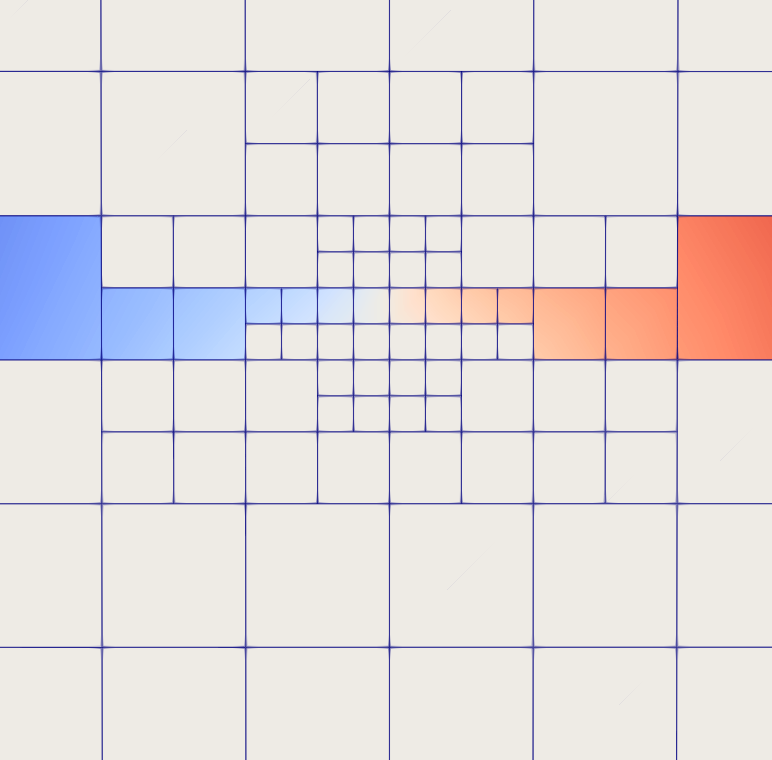}}\hfill
  \subfloat[north pole, cycle=14]{%
    \includegraphics[width=.24\textwidth]{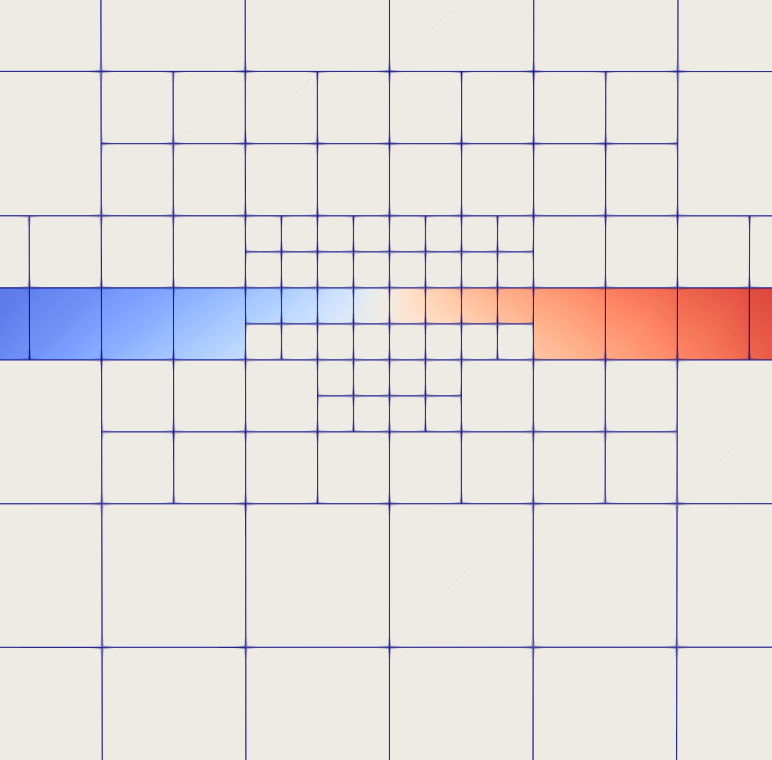}}\hfill
  \subfloat[north pole, cycle=15]{%
    \includegraphics[width=.24\textwidth]{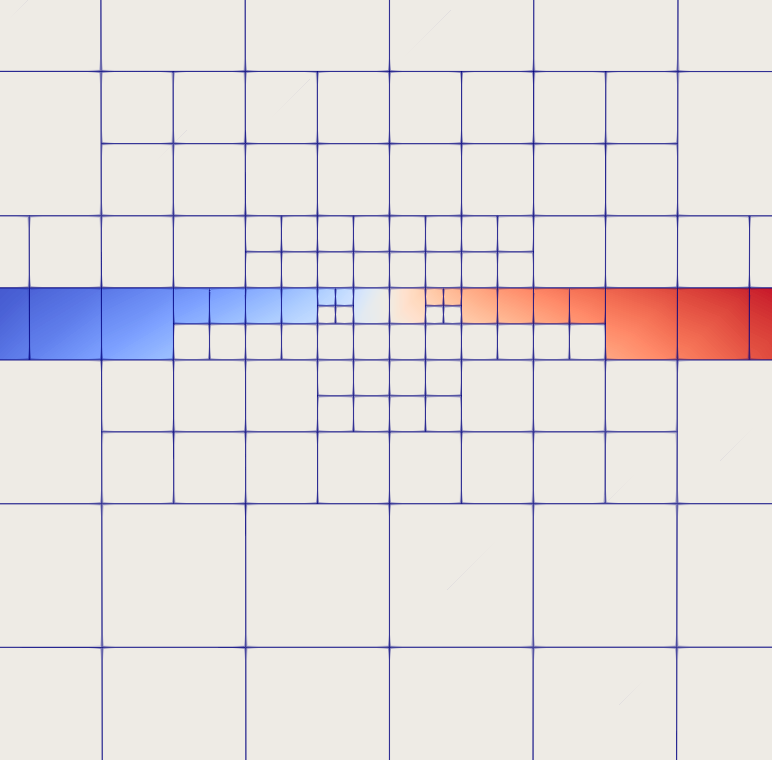}}\\
  \caption{Snapshots of the mesh crosscuts at different cycles of the adaptive procedure from Section~\ref{sec:adaptiveTrace}. The surface $\Gamma_h$ is not shown. Active elements $\mathcal{T}_h^\Gamma$ and the corresponding domain $\omega_h$ are colored by the values of the solution \eqref{polar} with $\lambda=0.4$. Vertical direction corresponds to OZ axis. Top: the whole domain $[-2,2]^3$, with many cells remain coarse throughout the procedure. Bottom: closeup view of the north pole $(0,0,1)$ of the unit sphere where the gradient of the solution \eqref{polar} blows up.}\label{fig:snapshots}
\end{figure}

\subsection{Unfitted quadratures and other implementation details}
The adaptive stabilized TraceFEM scheme of Section~\ref{sec:adaptiveTrace} was implemented in the Finite Element library \textbf{deal.II} \cite{dealII95,dealii2019design}. Since the method is not standard, we start with discussing some implementation details. 

\begin{itemize}
    \item The degrees of freedom of the level-set function exist across the entire mesh domain, whereas the degrees of freedom of the solution are confined to the colored, active domain of intersected cells. In principle, the discrete level-set approximation could have a different order or even an independent mesh from that of the solution. However, for the sake of convenience, we utilized the same triangulation for both the solution and the level-set in our implementation.
    
    \item Given that the mesh contains hanging nodes, ensuring the continuity of the FE spaces defined on it is necessary for a $H^1$-conforming method. This continuity requirement extends to both the discrete level-set and the discrete solution. To achieve this, we express the continuity condition for each hanging node as a linear combination involving local degrees of freedom, which is subsequently incorporated into the linear system. We apply a similar post-processing technique to the discrete level-set function, defined by a point-wise Lagrange interpolant, to eliminate any gaps in the discrete surface $\Gamma_h$.
    
    \item The implementation of \eqref{linear_system} requires the integration of polynomial functions over the intersections of the implicit surface $\Gamma_h$ with end cells from $\T_h^\Gamma$. This procedure is non-standard, and our implementation relies on the dimension-reduction approach detailed in \cite{saye2015high}. Notably, this algorithm is purpose-built for quadrilaterals and can accommodate higher-order approximations of $\Gamma_h$.

    \item Implementation of  stabilization forms $s_h^{NV}$ and $s_h^{JF}$ requires standard, e.g. Gauss--Lobatto, quadratures on a three-dimensional cube $S_T$ and on a two-dimensional square $F$, correspondingly. 
    
    \item Computation of the indicator \eqref{total_indicator} involves the same numerical integration procedures as used for \eqref{linear_system}.
       
    \item Although the forcing term $f^e$ is not an $L^2(\Gamma_h)$ function, the integral on the right-hand side of \eqref{linear_system} remains well-defined, provided that none of the surface quadrature nodes intersect the north or south poles when projected onto $\Gamma$.
    
    \item  In the course of adaptive refinement some of  inactive cells and some active cells not from $\T_h^\theta$ are refined so that the mesh remains graded.
\end{itemize}

\subsection{Uniform refinement}
 The first example serves to motivate the adaptivity and to test our implementation of TraceFEM for $V_h^1$ and $V_h^2$ ambient spaces. We choose  the exact solutions \eqref{polar}, $u_\lambda\in H^{1+\lambda}(\Gamma)$  with $\lambda=1.0$, $\lambda=0.7$ and $\lambda=0.4$ and solve the discrete problems \eqref{FEM} with $k=1$, $s_h( u, v)=s_h^{NV}( u, v)$,
 and stabilization parameter $\rho_S=10h_S^{-1}$.  The active domain $\omega_h$ is refined uniformly and the obtained solutions $u_h\in V_h^1$ are compared with the normal extension $u^e$ of the exact solution $u\in H^{1+\lambda}(\Gamma)$. We evaluate the  following surface error norms,
 \begin{align}\label{error_norms}
     \|u_h-u^e\|_{L^2(\Gamma_h)}\,,\qquad \|\nabla_{\Gamma_h}{u_h}-(\nablaG{}u)^e\|_{L^2(\Gamma_h)}\,
 \end{align} and the results are presented in Figure~\ref{fig:uniform}. Optimal rates are observed for $\lambda=1.0$, which corresponds to $u\in H^2(\Gamma)$, but, as $\lambda$ decreases,  the  rates  deteriorate  in accordance with the regularity, $u_\lambda\in H^{1+\lambda}$, of the problem.  Asymptotically, the  rate $h^{\lambda}$ is attained for the energy norm as it would be expected for fitted FEMs.
 
 We conducted the same uniform refinement test using the gradient-jump face stabilization $s_h^{JF}$, and the results closely resemble those shown in Figure~\ref{fig:uniform}. Therefore, we have opted not to include an additional plot.
Next, we repeated the test for the $Q_2$ family with $k=2$ in $V_h^k$, employing the stabilizations $s_h^{NV}$ and $s_h^{JF2}$. \VY{When $\lambda=1$, the convergence rates are optimal and correspondent to a finite element space of second degree. However, in cases of low regularity where $\lambda<1$, the rate of convergence  attains $h^{\lambda}$ only in the energy norm.}
 
\def\varL{lambda=1/convergence_lambda=1.000000_STAB=0_IND=0.dat}
\def\varLL{lambda=7e-1/convergence_lambda=0.700000_STAB=0_IND=0.dat}
\def\varLLL{lambda=4e-1/convergence_lambda=0.400000_STAB=0_IND=0.dat}
 \begin{figure}
		\begin{subfigure}[b]{0.3\textwidth}            
			\begin{tikzpicture}[scale=0.8]
				\def\vara{5}
				\def\varb{10}
				\begin{loglogaxis}[ xlabel={dofs},xmax=1e6,xmin=1e2,  x tick label style={/pgf/number format/.cd,%
						scaled x ticks = false,
						set decimal separator={.},
						fixed}, ylabel={Error}, ymin=1E-4	, ymax=2
					,legend pos=outer north east]   
					
					\addplot+[black,mark=o, solid,mark size=3pt,mark options={solid},line width=0.5pt] table[x=Degrees_of, y=L2] {\varL}; 
					


     \addplot+[black,mark=diamond, solid,mark size=3pt,mark options={solid},line width=0.5pt] table[x=Degrees_of, y=L2] {\varLL}; 
					


     \addplot+[black,mark=+, solid,mark size=3pt,mark options={solid},line width=0.5pt] table[x=Degrees_of, y=L2] {\varLLL}; 
					

     
					\addplot[dashed,line width=1pt] coordinates { 
						(1,\vara) (1e1,\vara*0.1) (1e2,\vara*0.01) (1e3,\vara*0.001) (1e4,\vara*0.0001) (1e5,\vara*0.00001) (1e6,\vara*0.000001)
					};
					
					\def\sqrtten{0.316227766017}	
					\addplot[dotted,line width=1pt] coordinates { 
						(1,\varb) (1e1,\varb*\sqrtten) (1e2,\varb*0.1) (1e3,\varb*0.1*\sqrtten) (1e4,\varb*0.01) (1e5,\varb*0.01*\sqrtten) (1e6,\varb*0.001)
					};
				\end{loglogaxis}
			\end{tikzpicture}	
		\end{subfigure}%
  \hskip 2cm
		\begin{subfigure}[b]{0.3\textwidth}            
			\begin{tikzpicture}[scale=0.8]
				\def\vara{5}
				\def\varb{10}
				\begin{loglogaxis}[ xlabel={dofs},xmax=1e6,xmin=1e2,  x tick label style={/pgf/number format/.cd,%
						scaled x ticks = false,
						set decimal separator={.},
						fixed}, ymin=1E-2, ymax=2
					,legend pos=outer north east]   
					
					
					\addplot+[black,mark=o, solid,mark size=3pt,mark options={solid},line width=0.5pt] table[x=Degrees_of, y=H1] {\varL};


					
					\addplot+[black,mark=diamond, solid,mark size=3pt,mark options={solid},line width=0.5pt] table[x=Degrees_of, y=H1] {\varLL};      

					
					\addplot+[black,mark=+, solid,mark size=3pt,mark options={solid},line width=0.5pt] table[x=Degrees_of, y=H1] {\varLLL};

					\addplot[dashed,line width=1pt] coordinates { 
						(1,\vara) (1e1,\vara*0.1) (1e2,\vara*0.01) (1e3,\vara*0.001) (1e4,\vara*0.0001) (1e5,\vara*0.00001) (1e6,\vara*0.000001)
					};
					
					\def\sqrtten{0.316227766017}	
					\addplot[dotted,line width=1pt] coordinates { 
						(1,\varb) (1e1,\varb*\sqrtten) (1e2,\varb*0.1) (1e3,\varb*0.1*\sqrtten) (1e4,\varb*0.01) (1e5,\varb*0.01*\sqrtten) (1e6,\varb*0.001)
					};
					\legend{
						$\lambda=1.0$,
      						$\lambda=0.7$,
						$\lambda=0.4$,
						$\textrm{dofs}^{-1}$,
						$\textrm{dofs}^{-1/2}$
					}
				\end{loglogaxis}
			\end{tikzpicture}	
		\end{subfigure}%

		\caption{Uniform mesh refinement for different values of $\lambda$ using the scheme \eqref{FEM} which is based on the $Q_1$ TraceFEM and is stabilized by \eqref{s_nv}. Left: $\|u_h-u^e\|_{L^2(\Gamma_h)}$ error. Right: $\|\nabla_{\Gamma_h}u_h-(\nablaG{}u)^e\|_{L^2(\Gamma_h)}$ error. The exact solution $u_\lambda$ is of low regularity, $u\in H^{1+\lambda}(\Gamma)$ only.  The expected reduction of the convergence rates to $h^{\lambda}$, for the $H^1$-seminorm is observed for $\lambda<1$. The $L^2$-norm error appears to be less sensitive to $\lambda$ at least for the tested refinement levels. 
  }
     \label{fig:uniform}
	\end{figure}
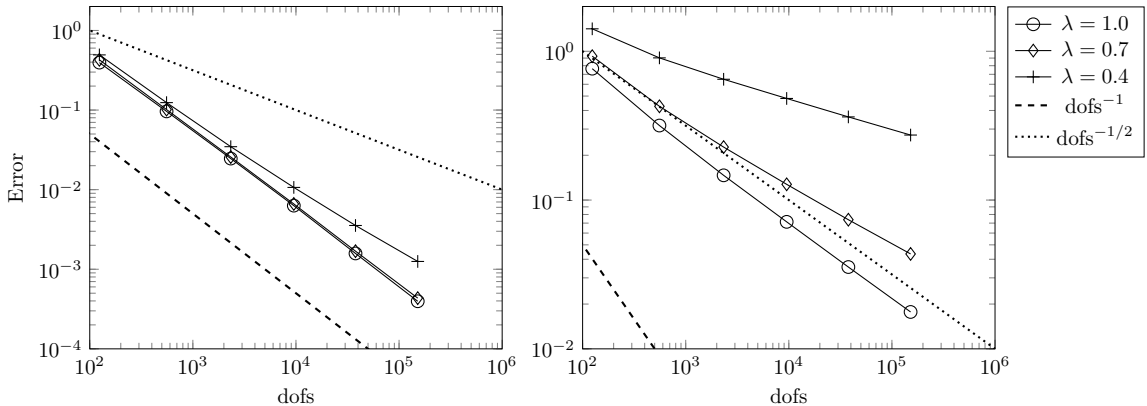

\subsection{Efficiency indexes}
\label{sec:efficiency}
In the numerical experiments we consider different notions of the efficiency. As usual, local efficiency indexes are computed for active cells $S_T \in \mathcal{T}_h^\Gamma$. These indices gauge how closely the actual error, $e_h=\nabla_{\Gamma_h}{u_h}-\nabla{}u^e$, is to the error indicator $\eta$ on the cell. Accumulated over all cells, a reliable indicator  estimates the error from above. The indicator is said to be efficient if the ratio of the indicator and the error, i.e. the efficiency index,  is bounded from above independent of the discretization level.

 We will consider three efficiency indexes which differ in the patch of neighboring cells contributing to the local error $e_h$ for the cell $S_T$. To compute the indexes, one maximizes the following ratios over all cuts $T=S_T\cap\Gamma_h$,
 \begin{align}\label{def::efficiency}
     I_1=\max\limits_{T}\,\frac{\eta(S_T)}{
							\|\xi_h\|_{\Gamma_h\cap\omega_{S_T}}}\,,\qquad
						I_2=\max\limits_{T}\,\frac{\eta(S_T)}{
							\|\xi_h\|_{\Gamma_h\cap\omega_T
						}}\,,\qquad
						I_3=\max\limits_{T}\,\frac{\eta_R(T)}{
							\|\xi_h\|_{\Gamma_h\cap S_T}	}
 \end{align}
 Here $\xi_h=\nabla_{\Gamma_h}{u_h}-\nabla{}u^e$ is the energy error, $\omega_{S_T}$ is the patch of all active cells from $\omega_h$ which share at least a vertex with the cell $S_T$; $\omega_T$ is the patch of all active cells from $\omega_h$ which share with the cell $S_T$ a face intersected by $\Gamma_h$. 
 Clearly, the efficiency index $I_3$ accumulates the error over a single cell $S_T$ only and it is the sharpest way to characterize the indicator.
 The notion of efficiency given by $I_3$ is too stringent, as it is known that the corresponding index blows up numerically even for a fitted FEM. At the same time, the theory of a fitted adaptive FEM guarantees that the indicator  is efficient if the error is accumulated over a patch of neighbors. This fact suggests that the indexes \VY{$I_1$  and $I_2$} are reasonable extensions of a similar notion to the unfitted finite element. The distinction between \VY{$I_1$ and $I_2$} lies in their dependence on the bulk mesh and the surface: in the former, the patch is based on the connectivity of the intersected cuts $T$, while in the latter, it relies on the connectivity of the bulk cells $S_T$.

 \begin{remark}\rm
 Note that the error part in \eqref{def::efficiency} does not include  the stabilization $s_h$ because we are interested in the surface error for a solution to a surface PDE. This is in contrast to the indicator $\eta(S_T)$ and to the natural discrete norm of \eqref{FEM} which include the stabilization $s_h$. One may question if adding the stabilization $s_h(u_h-u^e,u_h-u^e)$ to the denominator of indicators \eqref{def::efficiency} can lead to a  notion of efficiency which is more suitable to TraceFEM. As we found in our numerical experiments, such alternation does not  change  main conclusions drawn from the numerical experiments. For these reasons, we present the numerical results using the efficiency indexes as defined in \eqref{def::efficiency}.
 \end{remark}

\subsection{Efficiency and Reliability for the $Q_1$ elements}
\label{sec:Q1}

In this experiment, we assess the reliability and the efficiency of the indicator \eqref{numerical_indicator} using the $Q_1$ family of polynomials \eqref{q1_family}. Therefore, we choose a low-regularity solution \eqref{polar}, $u\in H^{1+\lambda}(\Gamma)$ with  $\lambda=0.4$, of the Laplace--Beltrami problem \eqref{LBw} posed on the unit sphere.  We run the adaptive TraceFEM  stabilized by $s_h=s_h^{JF}$ with $\sigma_F=10$ and by $s_h=s_h^{NV}$ with $\rho_S=10h_S^{-1}$ and evaluate surface errors \eqref{error_norms}. 

The numerical results, as presented in the top panel of Figure~\ref{fig:main_Q1}, confirm the a posteriori analysis conducted in Section~\ref{s:analysis}. Optimal rates are observed with both stabilizations, $s_h^{NV}$ and $s_h^{JF}$, as shown in Figure~\ref{fig:main_Q1}, \Green{and fewer degrees of freedom appear to be needed while using $s_h^{NV}$ to achieve comparable errors}.
Furthermore, in the plots of the bottom panel in Figure~\ref{fig:main_Q1}, we evaluate the efficiency indexes \eqref{def::efficiency} corresponding to several notions of efficiency discussed in Section~\ref{sec:efficiency}. The indexes $I_1$ and $I_2$ suggest the efficiency of the indicators for $Q_1$ adaptive TraceFEM.

	\def\varfnv{lambda=4e-1/convergence_lambda=0.400000_STAB=0_IND=1.dat}
	\def\varfjf{lambda=4e-1/convergence_lambda=0.400000_STAB=1_IND=1.dat}
	\begin{figure}[ht]
		\begin{subfigure}[b]{1\textwidth}            
			\begin{tikzpicture}[scale=0.7]
				\def\vara{450}
				\def\varb{30}
				
				\begin{loglogaxis}[ xlabel={dofs},xmax=1e6,xmin=1e2,  x tick label style={/pgf/number format/.cd,%
						scaled x ticks = false,
						set decimal separator={.},
						fixed}, ylabel={Error}, ymin=1E-4	, ymax=10
					,legend pos=outer north east]   
					
					\addplot+[blue,mark=+, solid,mark size=2pt,mark options={solid},line width=1.5pt] table[x=Degrees_of, y=L2] {\varfnv}; 
					
					\addplot+[green,mark=+, solid,mark size=1.5pt,mark options={solid},line width=1.5pt] table[x=Degrees_of, y=H1] {\varfnv};        
					
					\addplot+[red,mark=+, solid,mark size=1.5pt,mark options={solid},line width=1.5pt] table[x=Degrees_of, y=Estimator] {\varfnv};  
					
					
					\addplot[dashed,line width=1pt] coordinates { 
						(1,\vara) (1e1,\vara*0.1) (1e2,\vara*0.01) (1e3,\vara*0.001) (1e4,\vara*0.0001) (1e5,\vara*0.00001) (1e6,\vara*0.000001)
					};
					
					\def\sqrtten{0.316227766017}	
					\addplot[dotted,line width=1pt] coordinates { 
						(1,\varb) (1e1,\varb*\sqrtten) (1e2,\varb*0.1) (1e3,\varb*0.1*\sqrtten) (1e4,\varb*0.01) (1e5,\varb*0.01*\sqrtten) (1e6,\varb*0.001)
					};
				\end{loglogaxis}
			\end{tikzpicture}\begin{tikzpicture}[scale=0.7]
				\def\vara{2000}
				\def\varb{50}
				
				\begin{loglogaxis}[ xlabel={dofs},xmax=1e6,xmin=1e2,  x tick label style={/pgf/number format/.cd,%
						scaled x ticks = false,
						set decimal separator={.},
						fixed},  ymin=1E-4	, ymax=10
					,legend pos=outer north east]   
					
					\addplot+[blue,mark=+, solid,mark size=2pt,mark options={solid},line width=1.5pt] table[x=Degrees_of, y=L2] {\varfjf}; 
					
					\addplot+[green,mark=+, solid,mark size=1.5pt,mark options={solid},line width=1.5pt] table[x=Degrees_of, y=H1] {\varfjf};        
					
					\addplot+[red,mark=+, solid,mark size=1.5pt,mark options={solid},line width=1.5pt] table[x=Degrees_of, y=Estimator] {\varfjf};  
					
					
					\addplot[dashed,line width=1pt] coordinates { 
						(1,\vara) (1e1,\vara*0.1) (1e2,\vara*0.01) (1e3,\vara*0.001) (1e4,\vara*0.0001) (1e5,\vara*0.00001) (1e6,\vara*0.000001)
					};
					
					\def\sqrtten{0.316227766017}	
					\addplot[dotted,line width=1pt] coordinates { 
						(1,\varb) (1e1,\varb*\sqrtten) (1e2,\varb*0.1) (1e3,\varb*0.1*\sqrtten) (1e4,\varb*0.01) (1e5,\varb*0.01*\sqrtten) (1e6,\varb*0.001)
					};
					\legend{
						$\|e_h\|_{L^2(\Gamma_h)}$,
						$\|\nablaG e_h\|_{L^2(\Gamma_h)}$,
						$(\sum_T\eta^2(S_T))^{1/2}$,
						$\textrm{dofs}^{-1}$,
						$\textrm{dofs}^{-1/2}$
					}
				\end{loglogaxis}
			\end{tikzpicture}
		\end{subfigure}%

		\begin{subfigure}[b]{1.0\textwidth}            
			\begin{tikzpicture}[scale=0.7]
				\begin{loglogaxis}[ xlabel={dofs},xmax=1e6,xmin=1e2,  x tick label style={/pgf/number format/.cd,%
						scaled x ticks = false,
						set decimal separator={.},
						fixed}, ylabel={Efficiency}, 
					,legend pos=outer north east]   
					
					\addplot+[black,mark=none, solid,mark size=2pt,mark options={solid},line width=1.5pt] table[x=Degrees_of, y=Ultimate] {\varfnv}; 
					
					\addplot+[black,mark=none, dashed,mark size=2pt,mark options={solid},line width=1.5pt] table[x=Degrees_of, y=Surface] {\varfnv};

					\addplot+[black,mark=none, dotted,mark size=2pt,mark options={solid},line width=1.5pt] table[x=Degrees_of, y=Residual] {\varfnv}; 
				\end{loglogaxis}
			\end{tikzpicture}
			\begin{tikzpicture}[scale=0.7]
				\begin{loglogaxis}[ xlabel={dofs},xmax=1e6,xmin=1e2,  x tick label style={/pgf/number format/.cd,%
						scaled x ticks = false,
						set decimal separator={.},
						fixed},  
					,legend pos=outer north east]   
					
					\addplot+[black,mark=none, solid,mark size=2pt,mark options={solid},line width=1.5pt] table[x=Degrees_of, y=Ultimate] {\varfjf}; 
					
					\addplot+[black,mark=none, dashed,mark size=2pt,mark options={solid},line width=1.5pt] table[x=Degrees_of, y=Surface] {\varfjf};

					\addplot+[black,mark=none, dotted,mark size=2pt,mark options={solid},line width=1.5pt] table[x=Degrees_of, y=Residual] {\varfjf}; 
					\legend{$I_1$, $I_2$, $I_3$}
				\end{loglogaxis}
			\end{tikzpicture}	
		\end{subfigure}
		
		\caption{Adaptive refinement with $\theta=0.5$ using the indicator \eqref{total_indicator} for the $Q_1$ TraceFEM.  Left: $s_h^{NV}$ stabilization with $\rho_S=10h^{-1}_S$. Right: $s_h^{JF}$ stabilization with $\sigma_F=10$.  Top: surface errors \eqref{error_norms} for $e_h=u_h-u^e$ and the global estimator $(\sum_T\eta^2(S_T))^{1/2}$.  Bottom: efficiency indexes \eqref{def::efficiency} for different patches   of neighbors.
			The exact solution $u\in H^{1+\lambda}(\Gamma)$ with $\lambda=0.4$ is  given by \eqref{polar} on the unit sphere $\Gamma$.  We observe that the indicator \eqref{total_indicator} is reliable and efficient for $Q_1$ TraceFEM with both stabilizations.
   }
   \label{fig:main_Q1}
	\end{figure}
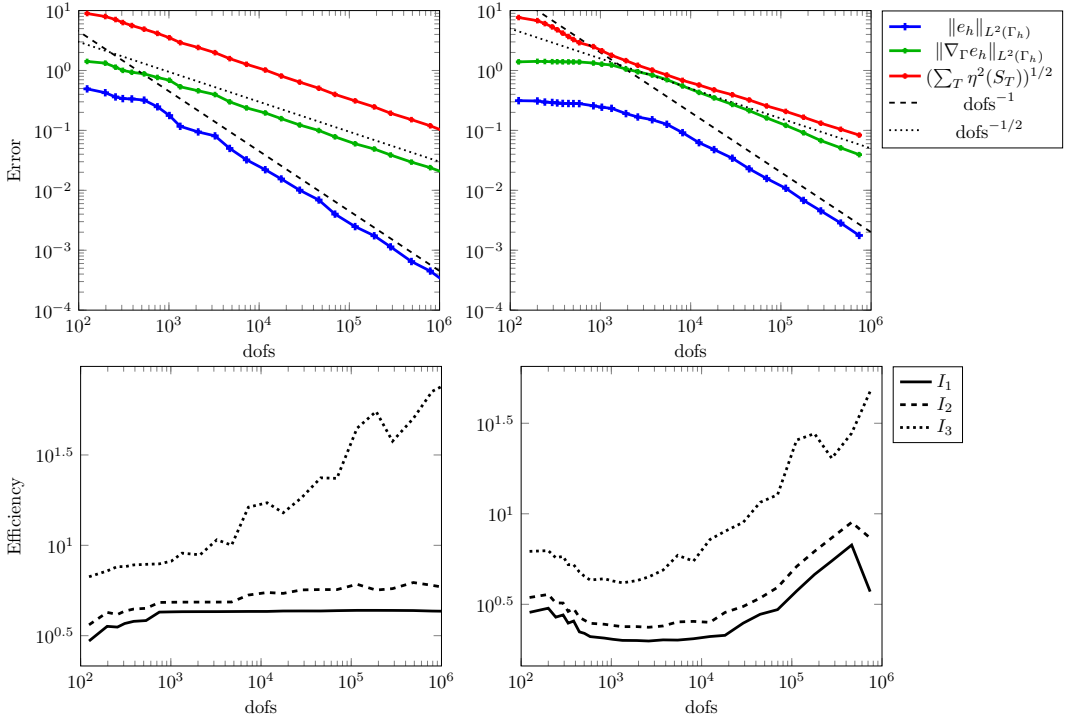

\subsection{Efficiency and Reliability for the $Q_2$ elements}
\label{sec:Q2}

We proceeded to repeat the experiment for the $Q_2$ TraceFEM, employing the discrete space $V_h^2$ for both the solution $u_h$ and the surface approximation $\Gamma_h$, following the same adaptive algorithm outlined in Section~\ref{sec:adaptiveTrace}. In this case, for the gradient-jump face stabilization, the $s_h^{JF}$ form was replaced by the $s_h^{JF2}$ form with $\sigma_F=10$.
As shown in the top panel of Figure~\ref{fig:main_Q2}, the $Q_2$ TraceFEM with gradient-jump face stabilization exhibits optimal convergence rates \Green{in the $L_2$ and $H_1$ norms}, while the $Q_2$ TraceFEM with normal-gradient volume stabilization \Green{shows almost optimal rates in the $H^1$ norm (which is the goal of the suggested indicator \eqref{total_indicator}) and suboptimal rates in the $L_2$ norm. Nevertheless, similar to the $Q_1$ case, the normal-gradient volume stabilization attains considerably smaller errors in both norms for the same number of unknowns}. Unlike the $Q_1$ scenario, the efficiency indexes in the $Q_2$ case exhibit linear growth with the number of degrees of freedom, as depicted in the bottom panel of Figure~\ref{fig:main_Q2}.

	\def\varfnv{Q2/no_s_h_eff_FE=2_convergence_lambda=0.400000_STAB=0_IND=1_stab_param=10.000000.dat}
\def\varfjf{Q2/no_s_h_eff_FE=2_convergence_lambda=0.400000_STAB=4_IND=1_stab_param=10.000000.dat}
	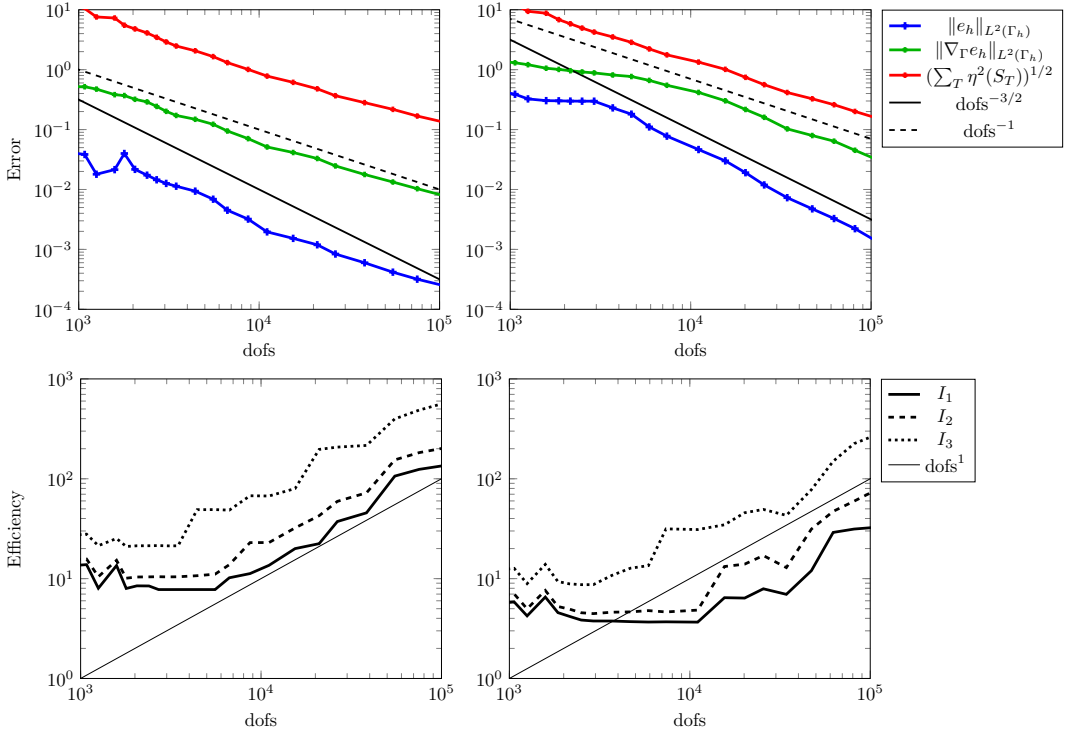
\begin{figure}
		\begin{subfigure}[b]{1\textwidth}            
			\begin{tikzpicture}[scale=0.7]
		\def\vara{1000}
		\def\varb{10000}
		\def\varc{1000}
		\def\vard{10000}
				\begin{loglogaxis}[ xlabel={dofs},xmax=1e5,xmin=1e3,  x tick label style={/pgf/number format/.cd,%
						scaled x ticks = false,
						set decimal separator={.},
						fixed}, ylabel={Error}, ymin=1E-4	, ymax=10
					,legend pos=outer north east]   
					
					\addplot+[blue,mark=+, solid,mark size=2pt,mark options={solid},line width=1.5pt] table[x=Degrees_of, y=L2] {\varfnv}; 
					
					\addplot+[green,mark=+, solid,mark size=1.5pt,mark options={solid},line width=1.5pt] table[x=Degrees_of, y=H1] {\varfnv};        
					
					\addplot+[red,mark=+, solid,mark size=1.5pt,mark options={solid},line width=1.5pt] table[x=Degrees_of, y=Estimator] {\varfnv};  
					
									
			
			\def\sqrtten{0.316227766017}	
			\addplot[solid,line width=1pt] coordinates { 
				(1,\varb) 
				(1e1,\varb*\sqrtten*0.1) (1e2,\varb*0.001) (1e3,\varb*0.0001*\sqrtten) (1e4,\varb*0.000001) (1e5,\varb*0.0000001*\sqrtten) (1e6,\varb*0.000000001)
			};
			
			\addplot[dashed,line width=1pt] coordinates { 
				(1,\varc) (1e1,\varc*0.1) (1e2,\varc*0.01) (1e3,\varc*0.001) (1e4,\varc*0.0001) (1e5,\varc*0.00001) (1e6,\varc*0.000001)
			};
			
			\end{loglogaxis}
			\end{tikzpicture}\begin{tikzpicture}[scale=0.7]
		\def\vara{50000}
		\def\varb{100000}
		\def\varc{7000}
		\def\vard{10000}
				\begin{loglogaxis}[xlabel={dofs},xmax=1e5,xmin=1e3,  x tick label style={/pgf/number format/.cd,%
						scaled x ticks = false,
						set decimal separator={.},
						fixed},  ymin=1E-4, ymax=10
					,legend pos=outer north east]   
					
					\addplot+[blue,mark=+, solid,mark size=2pt,mark options={solid},line width=1.5pt] table[x=Degrees_of, y=L2] {\varfjf}; 
					
					\addplot+[green,mark=+, solid,mark size=1.5pt,mark options={solid},line width=1.5pt] table[x=Degrees_of, y=H1] {\varfjf};        
					
					\addplot+[red,mark=+, solid,mark size=1.5pt,mark options={solid},line width=1.5pt] table[x=Degrees_of, y=Estimator] {\varfjf};  
					
					
			
			\def\sqrtten{0.316227766017}	
			\addplot[solid,line width=1pt] coordinates { 
				(1,\varb) 
				(1e1,\varb*\sqrtten*0.1) (1e2,\varb*0.001) (1e3,\varb*0.0001*\sqrtten) (1e4,\varb*0.000001) (1e5,\varb*0.0000001*\sqrtten) (1e6,\varb*0.000000001)
			};
			
			\addplot[dashed,line width=1pt] coordinates { 
				(1,\varc) (1e1,\varc*0.1) (1e2,\varc*0.01) (1e3,\varc*0.001) (1e4,\varc*0.0001) (1e5,\varc*0.00001) (1e6,\varc*0.000001)
			};
			
			\legend{
						$\|e_h\|_{L^2(\Gamma_h)}$,
						$\|\nablaG e_h\|_{L^2(\Gamma_h)}$,
						$(\sum_T\eta^2(S_T))^{1/2}$,
					$\textrm{dofs}^{-3/2}$,
					$\textrm{dofs}^{-1}$,
					}
				\end{loglogaxis}
			\end{tikzpicture}
		\end{subfigure}%

		\begin{subfigure}[b]{1.0\textwidth}            
			\begin{tikzpicture}[scale=0.7]
				\begin{loglogaxis}[ xlabel={dofs},xmax=1e5,xmin=1e3, ymin=1, ymax=1e3, x tick label style={/pgf/number format/.cd,%
						scaled x ticks = false,
						set decimal separator={.},
						fixed}, ylabel={Efficiency}, 
					,legend pos=outer north east]   
										\def\vara{0.001}
		\def\varb{100000}
		\def\varc{20000}
		\def\vard{10000}
  
					\addplot+[black,mark=none, solid,mark size=2pt,mark options={solid},line width=1.5pt] table[x=Degrees_of, y=Ultimate] {\varfnv}; 
					
					\addplot+[black,mark=none, dashed,mark size=2pt,mark options={solid},line width=1.5pt] table[x=Degrees_of, y=Surface] {\varfnv};

					\addplot+[black,mark=none, dotted,mark size=2pt,mark options={solid},line width=1.5pt] table[x=Degrees_of, y=Residual] {\varfnv}; 
     \addplot[solid,line width=0.2pt] coordinates { 
				(1,\vara) 
				(1e1,\vara*10)
				(1e2,\vara*100)
				(1e3,\vara*1000) (1e4,\vara*10000) (1e5,\vara*100000) (1e6,\vara*1000000)
    };
				\end{loglogaxis}
			\end{tikzpicture}
			\begin{tikzpicture}[scale=0.7]
				\begin{loglogaxis}[ xlabel={dofs},xmax=1e5,xmin=1e3, ymax=1e3, ymin=1, x tick label style={/pgf/number format/.cd,%
						scaled x ticks = false,
						set decimal separator={.},
						fixed},  
					,legend pos=outer north east]   
					\def\vara{0.001}
		\def\varb{100000}
		\def\varc{20000}
		\def\vard{10000}
					\addplot+[black,mark=none, solid,mark size=2pt,mark options={solid},line width=1.5pt] table[x=Degrees_of, y=Ultimate] {\varfjf}; 
					
					\addplot+[black,mark=none, dashed,mark size=2pt,mark options={solid},line width=1.5pt] table[x=Degrees_of, y=Surface] {\varfjf};

					\addplot+[black,mark=none, dotted,mark size=2pt,mark options={solid},line width=1.5pt] table[x=Degrees_of, y=Residual] {\varfjf}; 

     	\addplot[solid,line width=0.2pt] coordinates { 
				(1,\vara) 
				(1e1,\vara*10)
				(1e2,\vara*100)
				(1e3,\vara*1000) (1e4,\vara*10000) (1e5,\vara*100000) (1e6,\vara*1000000)
			};
   \legend{$I_1$, $I_2$, $I_3$, $\textrm{dofs}^{1}$}
				\end{loglogaxis}
			\end{tikzpicture}	
		\end{subfigure}
		
		\caption{Adaptive refinement with $\theta=0.5$ using the indicator \eqref{total_indicator} for the $Q_2$ TraceFEM. Left: $s_h^{NV}$ stabilization with $\rho_S=10h^{-1}_S$. Right: $s_h^{JF2}$ stabilization with $\sigma_F=\tilde{\sigma}_F=\tilde{\sigma}_\Gamma=\sigma_\Gamma=10$. Top: surface errors \eqref{error_norms} for $e_h=u_h-u^e$ and the global estimator $(\sum_T\eta^2(S_T))^{1/2}$.    
			The exact solution $u\in H^{1+\lambda}(\Gamma)$ with $\lambda=0.4$ is  given by \eqref{polar} on the unit sphere $\Gamma$. The indicator \eqref{total_indicator} is reliable \Green{in the energy norm} for the  $Q_2$ TraceFEM with both stabilizations.  The growth of all indexes shown on the bottom panels suggest the lack of efficiency. 
         \Green{Unlike the energy norm, for which the indicator was designed for, convergence rate in $L^2$ norm appears to be suboptimal for the  $s_h^{NV}$ stabilization.}
   }
   \label{fig:main_Q2}
	\end{figure}

\subsubsection{Effect of the stabilization parameter in $s_h^{JF2}$}
\label{sec:omitting}

It was observed in \cite{larson2020stabilization} that the performance of the stabilization $s_h^{JF2}$ defined in \eqref{s_jf2} is sensitive to the choice of the stabilization parameters. We would like to demonstrate how different values of $\sigma_F$ affect the adaptive TraceFEM with indicator \eqref{numerical_indicator}.

We did not observe improvements in efficiency by tuning the parameter $\sigma_F$ in Figure~\ref{fig:main_Q2}, where we used $\sigma_F=10$. To illustrate this point, we present the results of adaptive TraceFEM for two extreme values of the stabilization parameter: $\sigma_F=0.1$ and $\sigma_F=1000$, as shown in Figure~\ref{fig:Q2_stab_param}. Similar to Figure~\ref{fig:main_Q2}, the convergence rates are nearly optimal for both extreme values. However, when $\sigma_F=1000$, achieving the same level of accuracy requires more degrees of freedom compared to the case of $\sigma_F=0.1$.

This behavior of errors is consistent with what is typically observed during uniform refinement. In the adaptive setting, the indicator $\eta$ includes the stabilization, and when $\sigma_F=1000$, the estimator focuses on reducing the contribution of the stabilization $s_h(u_h,u_h)$ to the error functional $\tnorm{e_h}$, as illustrated in the right panels of Figure~\ref{fig:Q2_stab_param}.

\def\varfnv{Q2/no_s_h_eff_FE=2_convergence_lambda=0.400000_STAB=4_IND=1_stab_param=0.100000.dat}
\def\varfjf{Q2/no_s_h_eff_FE=2_convergence_lambda=0.400000_STAB=4_IND=1_stab_param=1000.000000.dat}
	\begin{figure}
		\begin{subfigure}[b]{1\textwidth}            
			\begin{tikzpicture}[scale=0.7]
		\def\vara{50000}
		\def\varb{5000}
		\def\varc{1000}
		\def\vard{10000}
				\begin{loglogaxis}[ xlabel={dofs},xmax=1e6,xmin=1e3,  x tick label style={/pgf/number format/.cd,%
						scaled x ticks = false,
						set decimal separator={.},
						fixed}, ylabel={Error}, ymin=1E-5	, ymax=100
					,legend pos=outer north east]   
					
					\addplot+[blue,mark=+, solid,mark size=2pt,mark options={solid},line width=1.5pt] table[x=Degrees_of, y=L2] {\varfnv}; 
					
					\addplot+[green,mark=+, solid,mark size=1.5pt,mark options={solid},line width=1.5pt] table[x=Degrees_of, y=H1] {\varfnv};        
					
					\addplot+[red,mark=+, solid,mark size=1.5pt,mark options={solid},line width=1.5pt] table[x=Degrees_of, y=Estimator] {\varfnv};  
					
									
			
			\def\sqrtten{0.316227766017}	
			\addplot[solid,line width=1pt] coordinates { 
				(1,\varb) 
				(1e1,\varb*\sqrtten*0.1) (1e2,\varb*0.001) (1e3,\varb*0.0001*\sqrtten) (1e4,\varb*0.000001) (1e5,\varb*0.0000001*\sqrtten) (1e6,\varb*0.000000001)
			};
			
			\addplot[dashed,line width=1pt] coordinates { 
				(1,\varc) (1e1,\varc*0.1) (1e2,\varc*0.01) (1e3,\varc*0.001) (1e4,\varc*0.0001) (1e5,\varc*0.00001) (1e6,\varc*0.000001)
			};
			
			\end{loglogaxis}
			\end{tikzpicture}\begin{tikzpicture}[scale=0.7]
		\def\vara{50000}
		\def\varb{2900000}
		\def\varc{60000}
		\def\vard{10000}
				\begin{loglogaxis}[xlabel={dofs},xmax=1e6,xmin=1e3,  x tick label style={/pgf/number format/.cd,%
						scaled x ticks = false,
						set decimal separator={.},
						fixed},  ymin=1E-5, ymax=100
					,legend pos=outer north east]   
					
					\addplot+[blue,mark=+, solid,mark size=2pt,mark options={solid},line width=1.5pt] table[x=Degrees_of, y=L2] {\varfjf}; 
					
					\addplot+[green,mark=+, solid,mark size=1.5pt,mark options={solid},line width=1.5pt] table[x=Degrees_of, y=H1] {\varfjf};        
					
					\addplot+[red,mark=+, solid,mark size=1.5pt,mark options={solid},line width=1.5pt] table[x=Degrees_of, y=Estimator] {\varfjf};  
					
					
			
			\def\sqrtten{0.316227766017}	
			\addplot[solid,line width=1pt] coordinates { 
				(1,\varb) 
				(1e1,\varb*\sqrtten*0.1) (1e2,\varb*0.001) (1e3,\varb*0.0001*\sqrtten) (1e4,\varb*0.000001) (1e5,\varb*0.0000001*\sqrtten) (1e6,\varb*0.000000001)
			};
			
			\addplot[dashed,line width=1pt] coordinates { 
				(1,\varc) (1e1,\varc*0.1) (1e2,\varc*0.01) (1e3,\varc*0.001) (1e4,\varc*0.0001) (1e5,\varc*0.00001) (1e6,\varc*0.000001)
			};
			
			\legend{
						$\|e_h\|_{L^2(\Gamma_h)}$,
						$\|\nablaG e_h\|_{L^2(\Gamma_h)}$,
						$(\sum_T\eta^2(S_T))^{1/2}$,
					$\textrm{dofs}^{-3/2}$,
					$\textrm{dofs}^{-1}$,
					}
				\end{loglogaxis}
			\end{tikzpicture}
		\end{subfigure}%
		
		\caption{The effect of the stabilization parameter $\sigma_F$ on the adaptive refinement in Figure~\ref{fig:main_Q2}. Left:  $\sigma_F=0.1$. Right: $\sigma_F=1000$. Surface errors \eqref{error_norms} for $e_h=u_h-u^e$ and the global estimator $(\sum_T\eta^2(S_T))^{1/2}$ are shown. We observe that decreasing the stabilization parameter does not improve the lack of efficiency while increasing it postpones the asymptotic regime of convergence. 
   }
   \label{fig:Q2_stab_param}
	\end{figure}
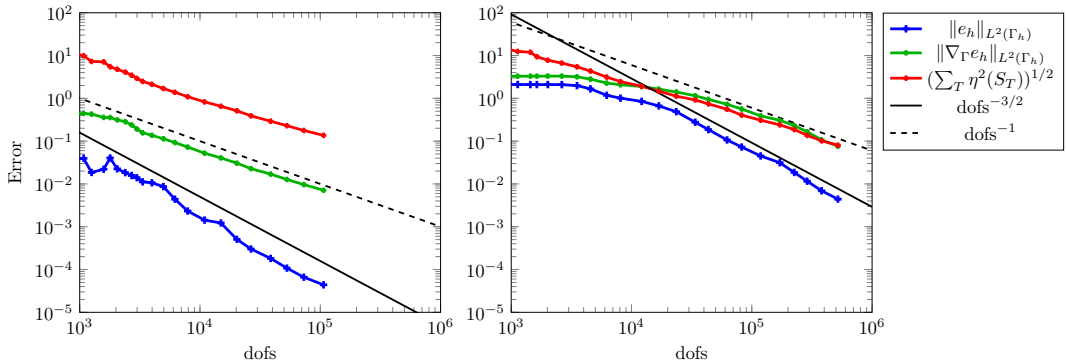

\section{Conclusions} \label{s_concl}
In this paper, we explore the application of adaptive stabilized TraceFEM for the first time. We focus on solving an elliptic problem on a fixed surface using the two lowest-order continuous finite element spaces based on $Q_1$ and $Q_2$ elements. For each family, we investigate both the gradient-jump face and normal-gradient volume stabilizations.

Our analysis demonstrates that the error indicator in the proposed adaptive TraceFEM is reliable, and our numerical tests confirm the theoretical findings. Specifically, for $Q_1$ elements, a reasonable choice for low-regularity solutions, we establish a robust and practical adaptive stabilized TraceFEM \Green{scheme}. 
\Green{In the case of $Q_2$}, the efficiency indexes grow proportionally with the number of active degrees of freedom.

Another significant contribution of this paper relates to the practical implementation of the proposed indicator. Rather than computing gradient jumps along one-dimensional curvilinear edges between surface patches, which can be computationally intensive due to the implicit surface description in TraceFEM, we evaluate gradient jumps on two-dimensional faces between bulk cells. This approach simplifies the implementation of the indicator. 

In conclusion, we recommend caution when using the $Q_2$ element in adaptive stabilized TraceFEM schemes, while the $Q_1$ element provides a highly robust adaptive method.

\subsection*{Acknowledgments}

The author T.H. was partially supported by the National Science Foundation 
Award DMS-2028346, OAC-2015848, EAR-1925575, and by the Computational
Infrastructure in Geodynamics initiative (CIG), through the NSF under Award
EAR-0949446, EAR-1550901, EAR-2149126 via the University of California -- Davis.
The author M.O. was partially supported by the  National Science Foundation under award DMS-2309197.
The author V.Y. was partially supported by the National Science Foundation Award OAC-2015848 and EAR-1925575.
Clemson University is acknowledged for generous allotment of compute time on Palmetto cluster.

\bibliographystyle{ieeetr}
\bibliography{citations}
\end{document}